%% file: spinstat3_b.tex
\documentclass[a4paper,12pt]{amsart}

\usepackage{amssymb,pstricks,
amscd}
\usepackage[all]{xy}
\usepackage[totalwidth=17cm,totalheight=24cm]{geometry}
\usepackage{graphicx}
\begin{document}

\input{macros}

\title{Particles with spin in stationary flat spacetimes}

\author[]{Thierry Barbot}
\address{Laboratoire d'analyse non lin\'eaire et g\'eom\'etrie\\
Universit\'e d'Avignon et des pays de Vaucluse\\
33, rue Louis Pasteur\\
F-84 018 AVIGNON
}
\email{thierry.barbot@univ-avignon.fr}
\author[]{Catherine Meusburger}
\address{Department Mathematik, FAU Erlangen-N\"urnberg\\Bismarckstrasse 1 1/2\\91054 Erlangen,
Germany}
\email{catherine.meusburger@math.uni-erlangen.de}
\thanks{T. B. was partially supported by CNRS, ANR GEODYCOS and A.N.R. program "Extensions of Teichmueller-Thurston theories (ETTT)", ANR-09-BLAN-0116-01. C.M.'s work is supported by the Emmy-Noether fellowship ME 3425/1-1 of  the German Research Foundation (DFG). Research visits during which work on this project was undertaken were also supported by the Emmy-Noether fellowship ME 3425/1-1.}

\keywords{}
\subjclass{83C80 (83C57), 57S25}
\date{\today}

\begin{abstract}
We construct stationary flat three-dimensional Lorentzian manifolds with singularities that  are obtained from Euclidean surfaces with cone singularities and closed one-forms on these surfaces. In the application to (2+1)-gravity, these spacetimes correspond to models containing massive particles with spin. We analyse their geometrical properties, introduce a generalised notion of global hyperbolicity and classify all  stationary flat spacetimes with singularities that are globally hyperbolic in that sense. We then apply our results to (2+1)-gravity and analyse the causality structure of these spacetimes in terms of measurements by observers. In particular, we derive a condition on observers that  excludes causality violating light signals despite the presence of closed timelike curves in these spacetimes.
\end{abstract}

\maketitle

%\tableofcontents

%\input{intro}

%\rem{You can of course modify the abstract as you wish. It's not very good, but I have no idea how to improve it. }

\section{Introduction}

Flat three-dimensional Lorentzian manifolds with conical singularities were first introduced in the physics literature on (2+1)-dimensional gravity, where they model  (2+1)-dimensional spacetimes  that contain massive point particles with spin.

The first models of   (2+1)-gravity with particles were derived in \cite{starus} and  \cite{djth}.  Their physical  properties and their quantisation were studied in the subsequent publications \cite{deser_scat, carlipscat, sg, thooft2, thooft3, thooft1}, which led to a large body of work on the classical aspects and quantisation of the models,  for an overview see \cite{carlip}.

As they are models that include matter and still are amenable to quantisation, these models play an important role in the research subject of quantum gravity. Since they allow one to investigate the quantisation of gravity coupled to matter,  they have been studied extensively in the physics literature.  Another reason why these models are of interest  in quantum gravity is their causality structure. It was shown in \cite{djth} that the presence of massive point particles with spin leads to  the presence of closed timelike curves in these models which, however, can be removed by excising a small cylinder around each particle.
Additionally, closed timelike curves can be generated dynamically when two spinless massive particles  approach each other with  sufficiently high speed (``Gott pairs''). A detailed investigation of this phenomenon has been given in \cite{gott,deser_stringtlike}, with the conclusion that these dynamically generated closed, timelike curves are not physically meaningful since they are present only for very short times, and, in particular,  ``time machines'' are excluded \cite{desjack_ttravel,deser_tttravel}.

Despite their relevance, many geometrical  properties of the models including massive point particles with spin are not fully understood even on the classical level. Although their properties have been investigated in the physics literature, they are very few results concerning the underlying mathematical structures.
The closest  treatment in the mathematics literature is the study of geometric Riemannian manifolds, mostly in the case of euclidean, spherical
or hyperbolic surfaces with conical singularities (\cite{troyanov3}, \cite{masur, matab})  sometimes in relation to the work on billiards,
and in the 3-dimensional case the work devoted to the Orbifold Theorem (\cite{cooper, porti}).

A similar treatment in the Lorentzian case is still at the beginning.
This includes in particular the causality issues arising in these models as well as the lack of systematic definitions and classification.  A systematic investigation of the mathematical features of three-dimensional Lorentzian spacetimes with particles has been initiated only recently and is mainly concerned with the case of constant negative curvature \cite{bb_hand,cone,corepart,  krasnovpart, collisions}.

In this article, we investigate flat stationary Lorentzian spacetimes with a general number of massive particles with spin. We  construct examples of  stationary flat Lorentzian spacetimes with particles that are based on  Euclidean surfaces with cone singularities and  closed one-forms on these surfaces. We introduce a generalised notion of global hyperbolicity that can be applied to these models despite the fact that they contain closed timelike curves. Based on this notion of global hyperbolicity, we classify the flat stationary globally hyperbolic Lorentzian
spacetimes with particles and give a detailed analysis of their geometrical properties.

The last section of the article is dedicated to a problem that is of high relevance to physics,
 namely the question,  how the presence of particles manifests itself in measurements by observers that probe the geometry of the spacetime by exchanging ``test lightrays''\footnote{The name ``test light rays'' is motivated by the fact that they play a role similar to the test masses used in general relativity.  They are lightlike geodesics in a given spacetime rather than actual lightlike point sources, and we neglect their impact on the stress energy tensor. This is different from the treatment in \cite{deser_lightlikesources}, which considers solutions of the Einstein equations in (2+1) dimensions with lightlike point sources.}. This idea is very natural from the viewpoint of general relativity, whose physical interpretation was formulated in terms of lightrays exchanged by observers from the beginning. It is also of special relevance to quantum gravity in four dimensions, as it is hoped that quantum gravity effects might manifest itself in cosmic microwave background radiation and thus be determined by means of lightrays. The (2+1)-dimensional models considered in this work share many properties with the cosmological models investigated in (3+1) dimensions.

 We show that light signals exchanged by observers
 correspond naturally to piecewise geodesic curves on the underlying Euclidean surfaces with cone singularities. We demonstrate how an observer can construct the relevant parameters that describe the spacetime from such measurements: the positions, masses and spins of such particles as well as their velocity with the respect to the observer.

 Building on these results, we investigate the causality issues associated with closed, timelike curves in spacetimes containing particles with spin.
  In particular, this allows one to establish a condition on the observers that excludes paradoxical signals, i.~e.~signals that are received before they are emitted. In physical terms, this implies that observers that stay away a sufficient distance from each particle, will not experience paradoxical light signals, even if the light signals themselves  enter a region around the particle which contains closed, timelike curves.

Note that these models do not involve the dynamically generated closed, timelike curves that are investigated in \cite{gott,deser_stringtlike,desjack_ttravel,deser_tttravel}, since the spacetimes under consideration are stationary. Instead, the spacetimes  investigated in this paper contain closed timelike curves  that are due to the presence of massive point particles with spin. Rather than investigating the impact of dynamically generated closed timelike curves and determining the spacetime regions in which they occur, we focus on
the impact of these closed timelike curves on observers at a sufficient distance from the particles and on light signals exchanged by such observers.

\section{The model: a single particle with spin} \label{singlepart}

%\subsection{Motivation}

\subsection{Definition}\label{defsec}

In the following we denote by $\R^{1,2}$ the three-dimensional Minkowski space and by
 $\Delta$ a timelike geodesic in $\R^{1,2}$.  We choose a suitable coordinate system $(x, y, t)$, in which  the Minkowski norm takes the form  $ds^2=dx^2+dy^2-dt^2$
 and $\Delta$ is given by the equation $x=y=0$.  We also introduce spatial polar coordinates $(r, \theta)$, $r\in\mathbb R^+, \theta\in[0,2\pi)$ which are given in terms of the spatial coordinates $x,y$ by $x=r\cos\theta$, $y=r\sin\theta$.

The model for a single particle introduced in \cite{djth} depends on two real parameters,
an angle  $0<\theta_0 < 2\pi$, in the following referred to as {\em deficit} or {\em apex angle}, and a  parameter $\sigma\in\RR$, in the following referred to as {\em spin}. The names for these parameters are motivated by their physical interpretation. It is shown in \cite{djth}, see also \cite{carlip},  that the metric for a single point particle in $\R^{2,1}$ is that of a cone with a deficit angle $\theta_0=2\pi-\mu$, where $\mu\in[0,2\pi)$ is the mass of the particle in units of the Planck mass. Moreover, it is shown there that the resulting spacetime has a non-trivial asymptotic angular momentum that is given by  $\sigma\in\mathbb R$. The parameter $\sigma$ which therefore is viewed as an internal angular momentum or spin of the particle in units of $\hbar$ \cite{djth}.

The associated flat Lorentzian manifold is constructed as follows. The equations $\theta=0$, $\theta=\theta_0$ define two timelike half-planes, which both have the timelike geodesic $\Delta$ as their boundary and which we denote,
respectively, $P_0$, $P_1$. These half-planes bound the region \begin{align}\label{wedge}W_{\theta_0} := \{ r>0, 0 \leq \theta \leq \theta_0 \},\end{align}
that we call a wedge of angle $\theta_0$. We glue $P_0$ to $P_1$ by the map $(r, 0, t) \mapsto (r, \theta_0, t+\sigma)$, which is a restriction of the
 elliptic isometry
$ g:(r, \theta, t) \mapsto (r, \theta+\theta_0, t+\sigma)$. The gluing of the wedge is pictured in Figure  \ref{wedge}.  The result is a manifold $M'_{\theta_0, \sigma}$,
that is naturally equipped with a flat Lorentz metric and homeomorphic to $\R^3$ minus a line. Note that the time orientation of
 Minkowski space induces a time orientation on $M'_{\theta_0, \sigma}$, namely the one for which the coordinate $t$ increases along
future oriented causal curves.
 \begin{figure}
  \centering
  \includegraphics[scale=0.4]{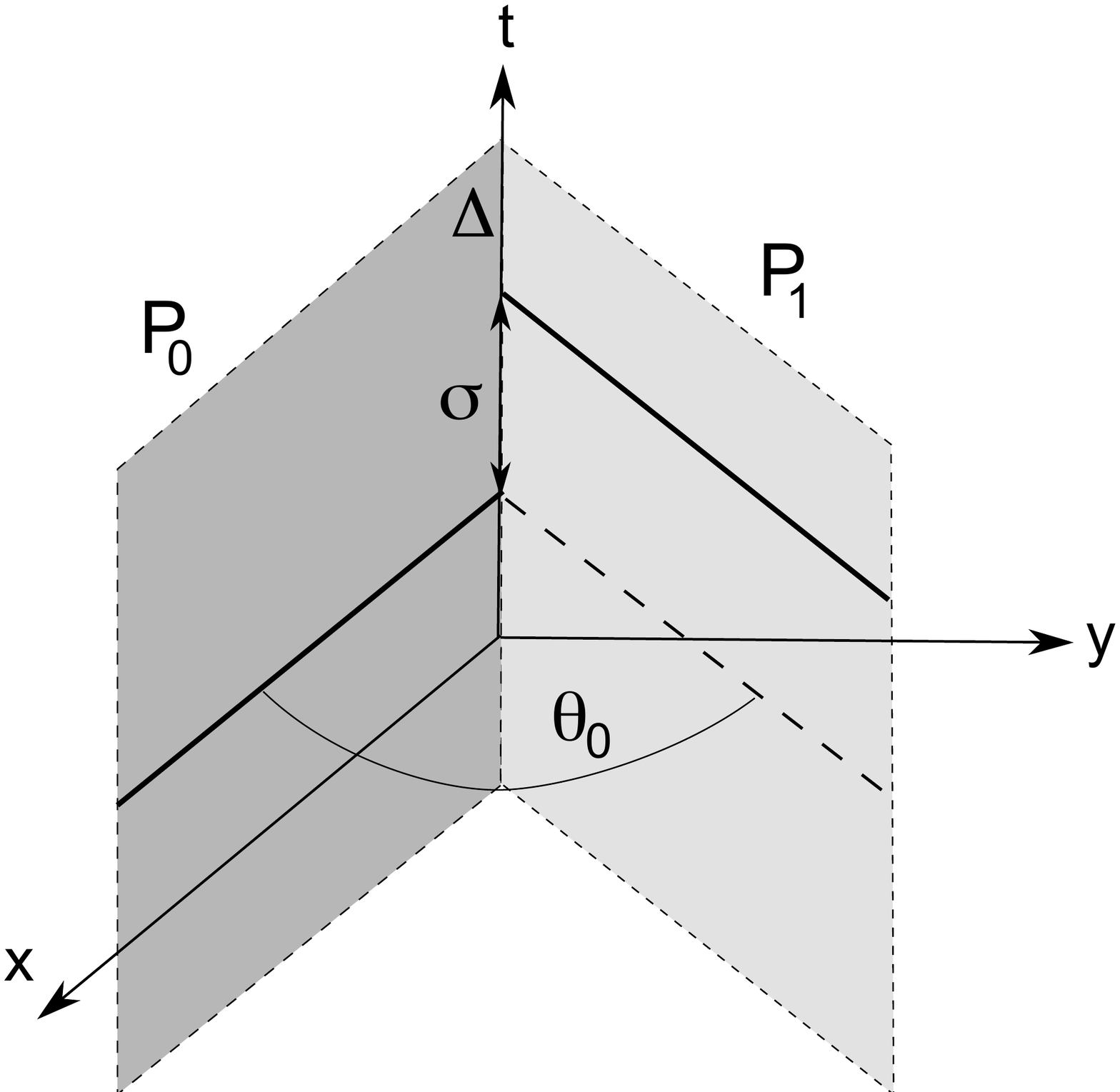}
\centering  \caption{Construction of the single particle-model by gluing a wedge $W_{\theta_0}$. The two solid  lines are identified by the elliptic isometry $(r, \theta, t) \mapsto (r, \theta+\theta_0, t+\sigma)$ that maps $P_0$ to $P_1$.}
  \label{wedge}
\end{figure}

We  now  demonstrate  how a (singular)   line can be added  to $M'_{\theta_0, \sigma}$ to obtain a manifold homeomorphic to $\R^3$. For this,
one is tempted to extend the gluing defined above to the closed wedge
$\overline{W}_{\theta_0}  := \{ r\geq0, 0 \leq \theta \leq \theta_0 \}$ in such a way that points  of the form $(0,0,t)$  are identified with
 $(0,0, t+\sigma)$. This is possible if  $\sigma=0$ and in that case yields a singular flat spacetime $M_{\theta_0, 0}$ that contains
a singular line characterised by the condition $r=0$. It corresponds to a (2+1)-dimensional spacetime with a single particle of mass $\mu=2\pi-\theta_0$ and vanishing spin $\sigma=0$.

However, this procedure does not work in the  case $\sigma\neq0$. For non-vanishing spin $\sigma$, the quotient of $\Delta$ by this isometry is a circle and not a line. When equipped with the quotient topology, it is no longer a manifold. Indeed, an open  disc in the $xy$-plane that is centred at the point $(0,0,0)$  corresponds to a
 union $\bigcup_{n=0}^\infty D_n$ of infinitely many circular sectors $D_n=\{ (r, \theta, t) \vert r<R_0, 0\leq \theta \leq \theta_0, t=n\sigma , n\in\Z\}$
that are identified along the line segments given by $\theta=0$ and $\theta=\theta_0$.

%\ques{should we include a picture that relates the new coordinates to the old ones?}
A more transparent description of spacetimes containing particles with non-vanishing spin is obtained by introducing a new set of coordinates  $(r, \alpha, \tau)$ that includes the  radial coordinate $r$ as well as
$$\tau=\theta_0t-\sigma\theta\qquad
\alpha=\frac{2\pi}{\theta_0}\theta.$$
As the coordinate $\alpha$  has the range  $0\leq \alpha<2\pi$, it induces
 a map
$\varphi: \R^3 \setminus \{r=0\} \rightarrow M'_{\theta_0, \sigma}$.
The pull-back by $\varphi$ of the flat metric on $M'_{\theta_0, \sigma}$ to $\mathbb R^3\setminus\{r=0\}$ is given by
\begin{align}\label{pback}
ds^2=-\frac1{\theta_0^2}d\tau^2 - \frac{\sigma}{\pi\theta_0}d\alpha d\tau + \frac{r^2\theta_0^2 - \sigma^2}{4\pi^2}d\alpha^2 + dr^2.\end{align}
In the following we denote by $M_{\theta_0, \sigma}$ the manifold $\R^3$ equipped with this metric  outside
the singular line given by $r=0$. % which  corresponds to a particle with mass $m=2\pi-\theta_0$ and spin $\sigma$.
It  contains (an isometric copy of)
$M'_{\theta_0, \sigma}$.
Note that this formula can be extended  to the case $\theta_0 \geq 2\pi$. In geometrical terms, this amounts to the following construction.
We consider the wedge $W_{\theta_0}$ not as embedded in  Minkowski space, but as embedded in the  universal cover of $\R^{1,2} \setminus \Delta$. In other words,
we introduce a  coordinate system $(r, \theta, t)$, where  $\theta$ is no longer defined modulo $2\pi$ but now parametrises
the entire real line. The resulting flat singular spacetime $M_{\theta_0, \sigma}$  is then given as a $n$-branched cover along
the singular line over $M_{\theta_0/n, \sigma}$, where $n$ is chosen so that $\theta_0/n$ is less than $2\pi$.
In this description,  the mass parameter $\mu=2\pi - \theta_0$ can become negative or vanish. In particular,  the limit case $\theta_0=2\pi$ yields a massless particle  with non-vanishing  spin $\sigma$.

\subsection{Closed timelike curves (CTCs) and the CTC surface}\label{ctcsec}

%\rem{I replaced ``critical radius'' and ``critical surface'' by ``CTC radius'' and ``CTC surface''. The word critical could give physicists a wrong impression. It is mostly used in relation %with violent phenomena such as collapse, explosions, liquids starting to boil,...etc.}
In contrast to the coordinate $t$, the coordinate $\tau$ on $M_{\theta_0, \sigma}$ is not a time function. Introducing the variable  $r_0=\sigma/\theta_0$, we can rewrite the metric \eqref{pback}
as
\begin{align}\label{metric2} ds^2=-\frac1{\theta_0^2}\,d\tau^2 - \frac{r_0}{\pi}\,d\alpha d\tau + \left(\frac{\theta_0}{2\pi}\right)^2\!\!(r^2-r_0^2)\,d\alpha^2 + dr^2\end{align}
For a given value of $\tau$, the circle $C_{\tau,r}=\{(r,\alpha,\tau)\,|\, 0\leq \alpha\leq 2\pi\}$ of constant radius $r$ is spacelike if $r>\vert r_0 \vert$, timelike if $r<\vert r_0 \vert$, and null  for $r=\vert r_0 \vert$. This implies in particular that it defines a closed timelike curve  (CTC) for $r<|r_0|$ and a closed lightlike curve for $r=|r_0|$. In the following we will therefore refer to $|r_0|$ as the \textit{CTC radius}, to the surface $H=\{(|r_0|,\alpha, \tau)\,|\, 0\leq \alpha\leq 2\pi, \tau\in\R\}$ of constant radius $|r_0|$ as the \textit{CTC surface}.  We call the domain $U=\{(r,\alpha,\tau)\,|\, 0<r<|r_0|, 0\leq \alpha\leq 2\pi, \tau\in\R\}$ the \textit{CTC region} and
the region $M^*_{\theta_0,\sigma}=\{(r,\alpha,\tau)\,|\, r>|r_0|,  0\leq \alpha\leq 2\pi, \tau\in\R\}$ the \textit{interior region} of the spacetime. The latter is a manifold with boundary,
whose boundary is the CTC surface $H$. It is the complement of the CTC region, which is diffeomorphic to  $\D^2 \times \R$, where $\D^2$ denotes the open disc in $\R^2$.
On the CTC surface the metric  \eqref{metric2} degenerates to $$ds^2|_{H}=-\frac1{\theta_0^2}\,d\tau^2 - \frac{r_0}{\pi}\, d\alpha d\tau = -\left(\frac1{\theta_0^2}\,d\tau + \frac{r_0}{\pi}\,d\alpha\right)\, d\tau.$$ Note  that $d\tau$
does not vanish along spacelike curves in $H$. It follows that a non-timelike curve in $H$ cannot close up unless
 it is contained in a circle in $H$ characterised by the condition  $\tau=$ constant. Such circles are
lightlike but they are not geodesics. In the following we will call them  \textit{null circles} on the CTC surface.  Note  that the future of a point $x=(\vert r_0 \vert, \alpha, \tau_0)$ in the CTC surface $H$, i.~e.~ the points in $H$ that can be connected to $x$ via a future directed timelike curves in $H$,  is the region $x_+=\{ (r_0,\alpha, \tau)\,|\, \tau > \tau_0 \}\subset H$ above the null circle containing $x$. The future in $H$ of a point on a given null circle  therefore coincides with the future in $H$ of this null circle.

The CTC region $U$ contains many closed timelike curves (CTCs). Note, however, that it does not contain closed timelike geodesics. It follows from the expression for the metric, that in order to close up,   timelike curves  must have an acceleration, which is related to the spin parameter $\sigma$. The smaller  the value of the spin parameter, the bigger the acceleration associated with CTCs must be, and it tends to infinity in the limit of vanishing spin.
Due to the presence of CTCs, the CTC region $U$ exhibits quite pathological causality relations. The future (or the past) inside $U$ of any point in $U$ is the entire CTC region. Its future
 (or past) in $M'_{\theta_0, \sigma}$ is the entire manifold  $M'_{\theta_0, \sigma}$.

In contrast to the CTC region, the causality structure of the interior region $M^*_{\theta_0, \sigma}$ is well-behaved. As the  coordinate $\tau$ defines a time function on $M^*_{\theta_0, \sigma}$, $M^*_{\theta_0, \sigma}$ contains no CTCs.  Of course, this does not exclude that a timelike curve starts in the interior region,
enters the CTC region and then returns to its starting point in the  interior region.  However, the absence of CTCs in the interior region implies that any closed timelike curve with a starting point in the interior region $M'_{\theta_0, \sigma}$ must enter the CTC region.

\subsection{Killing vector fields}\label{killvecsec}
The group of time orientation and orientation preserving isometries of $M'_{\theta_0, \sigma}$ is an abelian group of dimension two. It is generated by rotations  $(r,\alpha,\tau)\mapsto (r,\alpha+\alpha_0,\tau)$ and by translations   $(r,\alpha,\tau)\mapsto(r,\alpha, \tau+\tau_0)$. In particular,
$M'_{\theta_0, \sigma}$ is stationary: the translation along $\tau$ induces an isometry between the level sets of
$\tau$. However, if the spin $\sigma$ is nonzero, $M_{\theta_0, \sigma}$ is not static because the lapse term $ - \frac{r_0}{\pi}d\alpha d\tau$ in \eqref{metric2}
does not vanish.

The CTC region, CTC surface and the interior region are distinguished by the Killing vector $\partial_\alpha$ associated with the rotations. The CTC region is characterised by the condition that $\partial_\alpha$ is timelike, the CTC surface  is the locus where
$\partial_\alpha$ is lightlike, and the interior region is the region where $\partial_\alpha$ is spacelike.

\subsection{Cauchy surfaces}\label{csurfacesec}

As the CTC region around the particles contains closed timelike curves, $M'_{\theta_0, \sigma}$ is far from being globally hyperbolic.
However,  the level surfaces of the coordinate $\tau$ are Cauchy surfaces for the interior region in the sense that any
inextendible causal curve in $M'_{\theta_0, \sigma}$ that is \textit{contained in $M^*_{\theta_0, \sigma}$} must intersect
every level set of $\tau$. We express this property by saying that \textit{$M^*_{\theta_0, \sigma}$ is globally hyperbolic
relatively to its boundary}.
As observed in Section \ref{ctcsec},   the only non-timelike loops in the CTC surface $H=\partial M^*_{\theta_0, \sigma} $ are the null circles.
The boundary of any Cauchy surface in the interior region  $M^*_{\theta_0, \sigma}$ therefore must coincide with one of  these null circles.

\subsection{The developing map}\label{devmap}
The universal covering  $\widetilde{M}'_{\theta_0, \sigma}$ of  $M'_{\theta_0,\sigma}$ is homeomorphic to the manifold obtained by
by taking an infinite number of copies $W^i_{\theta_0},$ $i \in \Z$, of the wedge $W_{\theta_0}$ introduced in Section \ref{defsec} and gluing them along the associated planes $P_0^i$, $P_1^i$ via the the elliptic isometry $g: (r, \theta, t) \mapsto (r, \theta+\theta_0, t+\sigma)$:
$
P_1^i\sim P_0^{i+1}
$ for all $i\in\Z$.
The covering map
$p: \widetilde{M}'_{\theta_0, \sigma} \to M'_{\theta_0, \sigma}$ is the map induced by
  the isometry $W^i_{\theta_0} \cong W_{\theta_0}$.
Denote by $g^i$ the elliptic isometry obtained by applying the elliptic isometry $g$ % that identifies the sides of the wedge
$i$ times:  $g^i: (r, \theta, t) \mapsto (r, \theta+i \theta_0, t+i \sigma)$. Then the
 maps $W^i_{\theta_0} \to g^i(W_{\theta_0})$ together define a (local) isometry
$\cD: \widetilde{M}'_{\theta_0, \sigma} \to \R^{1,2} \setminus \Delta$. This map is the developing map of the Minkowski structure
on $M'_{\theta_0, \sigma}$. It is equivariant with respect to the natural actions of $\pi_1(M'_{\theta_0, \sigma}) \approx \Z$ on $\widetilde{M}'_{\theta_0, \sigma}$
and on $\R^{1,2} \setminus \Delta$.
The first action is the one that maps every $W^i_{\theta_0}$ onto $W^{i+1}_{\theta_0},$ and the second action is the one induced by $g$.

Note that the map $\cD$ is \textit{never} a homeomorphism. When $|i|$ increases, the wedges $p(W^i_{\theta_0}) = g^i(W_{\theta_0})$ wrap around the line
$\Delta$, and for  $|i|> 2\pi/\theta_0$ overlap with the initial wedge $W_{\theta_0}$. This overlapping is a perfect matching if and only if $2\pi/\theta_0$ is rational,
in which case $M'_{\theta_0, \sigma}$ might be seen as a finite quotient of $ \R^{1,2} \setminus \Delta$.  This reflects a general pattern that is also present in the case of Minkowski spacetimes with multiple particles. The developing maps for these spacetimes
 are not one-to-one. Moreover,
as we will see in the following, the developing maps of spacetimes with at least two particles  are surjective.  The developing maps are thus quite pathological, which reflects the fact that the regular part of these manifolds {\em cannot} be obtained as a
quotient of a region of the Minkowski space.

\subsection{Geodesics}
\label{sub:geodesic}

 \begin{figure}
  %\centering
  \includegraphics[scale=0.4]{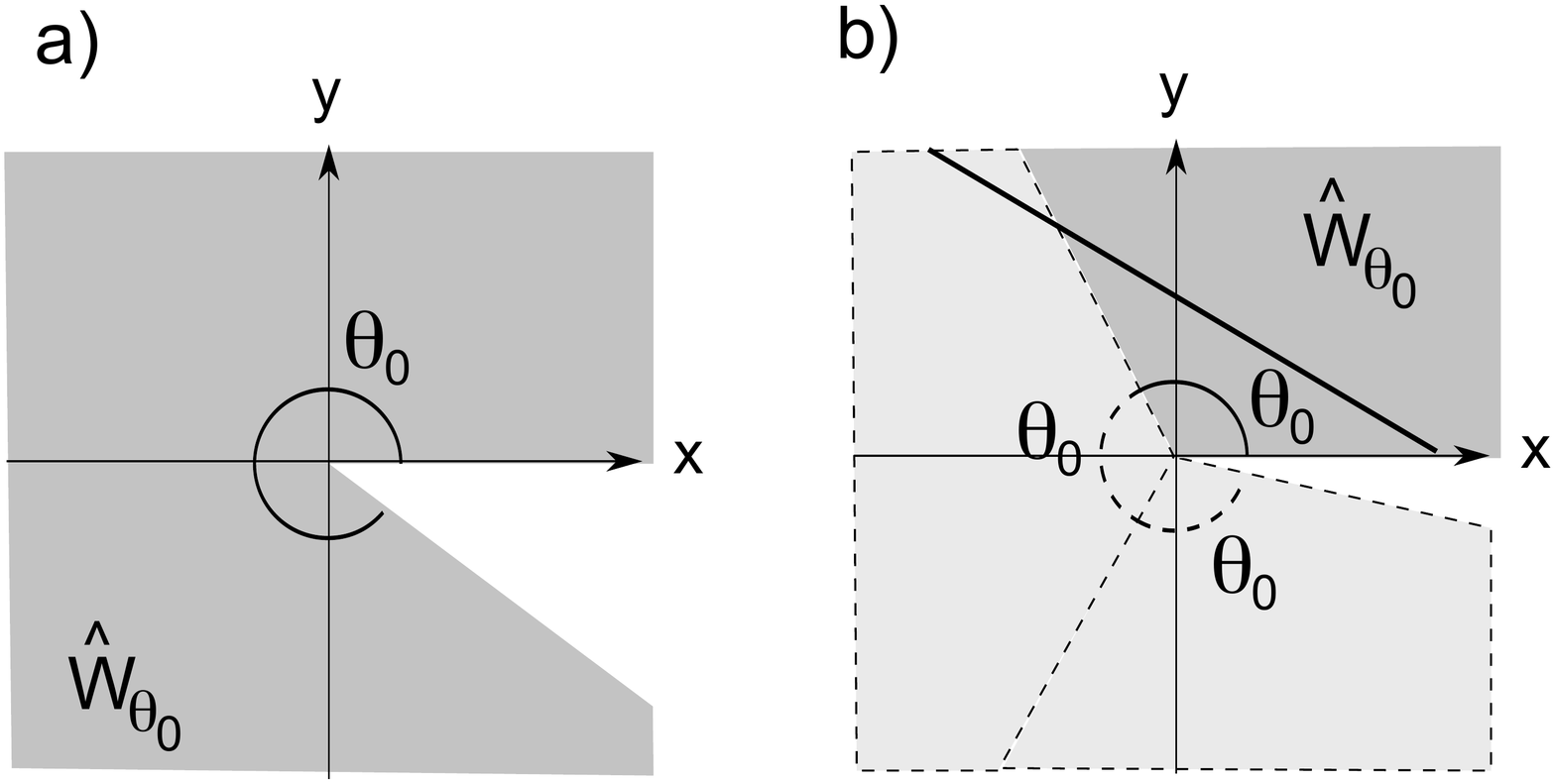}
\centering  \caption{The defining wedge of the Euclidean plane with a cone singularity for a) $\theta_0>\pi$ and b) $\pi/2< \theta_0<\pi$. The solid line in b) corresponds to a geodesic loop in $\R^2_{\theta_0}$ with winding number 1.}
  \label{devplane}
\end{figure}

To investigate the properties of the geodesics in $M'_{\theta_0,\sigma}$, it is useful to introduce the {\em Euclidean plane with a cone singularity} of cone angle $\theta_0$, which in the following will be denoted by  $\R^2_{\theta_0}$.
The definition is analogous to the one of the manifold $M_{\theta_0,\sigma}$.
Consider the wedge of angle $\theta_0$ in the Euclidean plane $\R^2\setminus\{0\}$:
$
\hat{W}_{\theta_0}=\{ (r,\theta)\,|\, r>0, 0\leq \theta\leq \theta_0\}
$
and glue the two sides of this wedge via the identification $(r,0)\sim (r,\theta_0)$. Alternatively, the Euclidean plane with a cone singularity is obtained as the completion of
the following metric on $\R^2\setminus \{0\}$  given in polar coordinates
\begin{align}\label{planemetric}
ds_E^2=\left(\frac{\theta_0}{2\pi}\right)^2r^2d\theta^2 + dr^2.
\end{align}
The vertical projection $p_0: \R^{1,2} \setminus \Delta \to \R^2 \setminus \{0\}$ then induces a map $M_{\theta_0, \sigma} \to \R^2_{\theta_0}$.
Denote by $\pi_0: \widetilde{M}'_{\theta_0, \sigma} \to \R^2_{\theta_0}$, $\pi_0=p_0\circ \cD$ the composition of this projection with the developing map.
Let now $c: (a,b) \to M'_{\theta_0, \sigma}$ a geodesic path (timelike, lightlike or spacelike). Then $c$ lifts to a geodesic path $\tilde{c}: (a,b) \to \widetilde{M}'_{\theta_0, \sigma}$.
As the developing map  $\cD$ is a local isometry, the image $\bar{c}:=\cD \circ \tilde{c}$ is  a geodesic path in $\R^{1,2} \setminus \Delta$ and its projection $p_0\circ \bar c=\pi_0\circ \tilde c$
 is a geodesic path in $\R^2_{\theta_0}$. Note that this path is constant if and only if  the geodesic $\bar c$ is parallel to $\Delta$.

 The path $\pi_0\circ \tilde c$ is a geodesic loop in $\R^2_{\theta_0}$  if and only if there exists a timelike geodesic $\Delta'$ parallel to $\Delta$ in $\R^{1,2}\setminus\Delta$ such that both its starting and endpoint of $\bar c$ lie on $\Delta'$.
 As we will see in the following, a lightlike geodesic $\bar c$ with this property corresponds to a returning lightray, i.~e.~a lightray sent out by an observer with worldline $\Delta'$  that returns to the observer at a later time.
This allows us to conclude that for $\theta_0\geq 2\pi$ there can be no returning lightray because $\R^2_{\theta_0}$ does not contain geodesic loops. Any geodesic in $\R^2_{\theta_0}$ lifts to a straight line in the Euclidean wedge $W_{\theta_0}\subset \R^2\setminus\{0\}$. If the angle $\theta_0$ is greater or equal to $\pi$, a straight line cannot intersect both sides of the wedge and hence cannot close. More generally, using the developing map for $\R^2_{\theta}$ and its identification with rotations in $\R^2\setminus\{0\}$ as shown in Figure \ref{devplane}
%in analogy to the discussion in Section \ref{devmap}
,  one finds that the existence  of a geodesic loop in $\R^2_{\theta_0}$
 with winding number $k$ around the cone singularity implies
$k\theta_0 < \pi.$

\subsection{CTC cylinders}
In the following section, we will extend our model obtain a more general notion of flat Lorentzian spacetimes  with a  particles. For this we will need to consider the interior region as a manifold with boundary that is given by the CTC surface. We introduce the following definition.

\begin{defi}
Let $h$ be a positive real number and $a\in\R$. A {\em CTC cylinder} of height $h$ based at $a$ is the region in $U$ between the two level sets $\tau^{-1}(a)$, $\tau^{-1}(a+h)$ of $\tau$.
The past (future) complete CTC cylinder based at $a$ is the past (future) in $U$ of the level set $\tau^{-1}(a)$.
\end{defi}

Note that all CTC cylinders for a given value of $h$ are isometric.  In contrast to the quantity $h$, the parameter $a$ therefore has no intrinsic geometrical meaning.
%The height $h$ has an intrinsic geometric meaning: ??
%it is the minimal proper time of a timelike curve in joining the bottom side
%$\tau^{-1}(a)$ to the top side $\tau^{-1}(b)$.
Similarly, all past  and future complete cylinders are isometric to the entire CTC region, which implies in particular that they
are \textit{complete}.

 \begin{figure}
  \centering
  \includegraphics[scale=0.4]{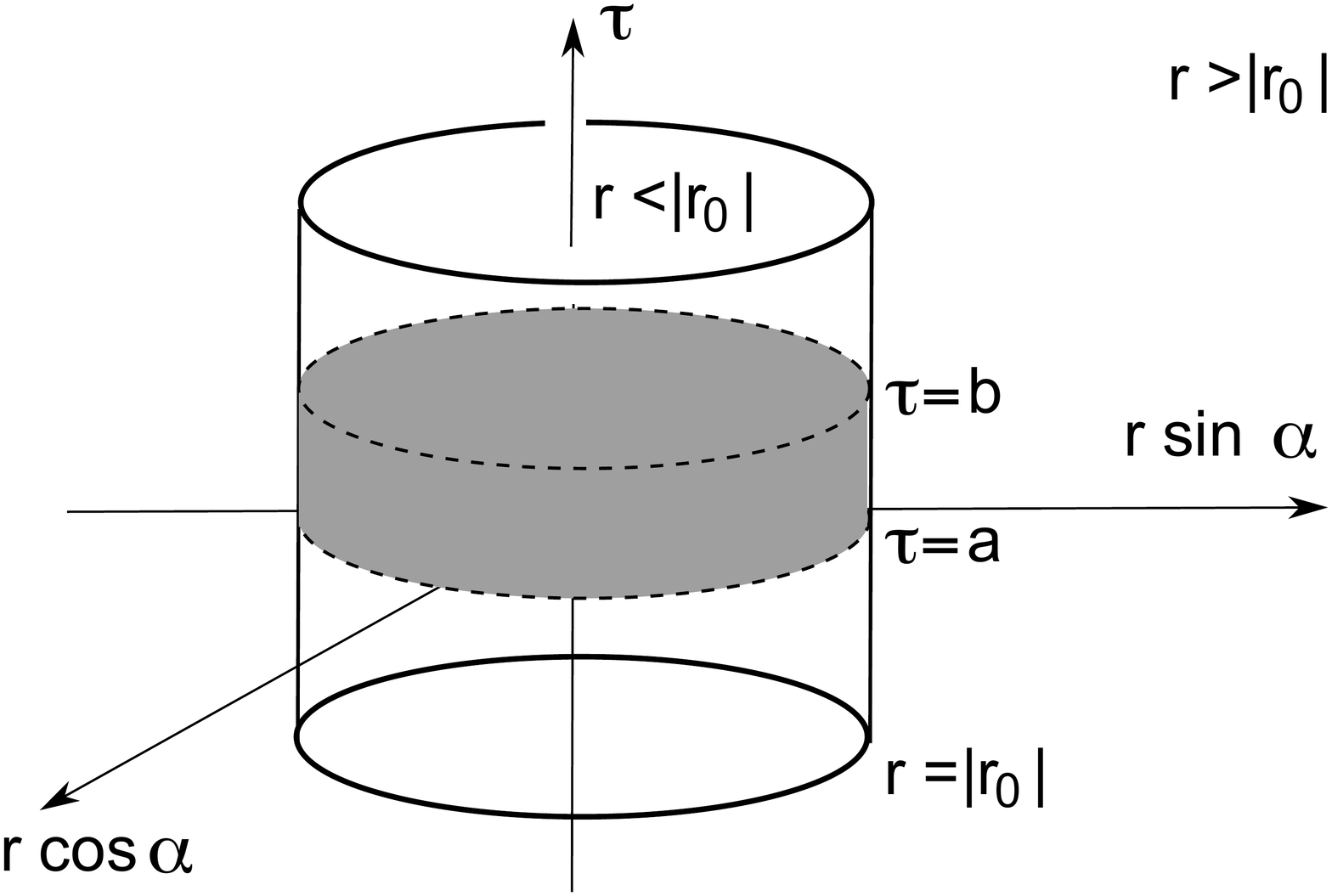}
\centering  \caption{CTC cylinder of height $h=b-a$ between the level sets $\tau^{-1}(a)$ and $\tau^{-1}(b)$ (shaded region).}
  \label{CTC}
\end{figure}

\section{Global hyperbolicity}

\subsection{Definition}

We are now ready to give a general definition of flat Lorentzian manifolds with particles
and to define a modified notion of global hyperbolicity, which will allow us to restrict the class of Lorentzian manifolds with particles under consideration.

 \begin{defi}\label{mfdef}
A flat Lorentzian manifold with particles is a  three-dimensional manifold $M$ with an embedded closed
$1$-submanifold $\Delta$ (not necessarily connected), such that $M \setminus \Delta$ is endowed with a flat Lorentzian metric and for every
$x$ in $\Delta$ there exists a neighbourhood $U$ of $x$ in $M$ such that $U \setminus \Delta$ is isometric to the neighbourhood
of a point on the singular line (the particle) in $M_{\theta_x, \sigma_x}$ with the singular line itself removed\footnote{Observe that the map $x \mapsto (\theta_x, \sigma_x)$ is then necessarily
locally constant on $\Delta$.}.
\end{defi}
%\rem{there was a notation problem here: $\theta_0$ and $\sigma$ must be allowed to vary with $x$ - otherwise  all particles have to have the same mass and spin}

%\ques{I think globally hyperbolic is not just a physics condition - this is linked to the mathematical properties of the spacetime, too. Its more an issue of Lorentzian vs Riemannian}

This definition provides us with a very general notion of a flat Lorentzian spacetime with particles  and thus potentially with a large class of examples.  However, there is no hope of obtaining a global understanding of flat spacetimes with particles without  suitable additional hypotheses. In Riemannian geometry, it is customary to impose as such an additional hypothesis the
 compactness of the ambient manifold. However, this condition is not suited to the Lorentzian context, since it implies issues with the causal structure such which are undesirable from both the mathematics and the physics point of view. Such issues arise even in the much simpler situation of flat Lorentzian manifolds without particles
 (cf. \cite{galloway, sanchezcompact}).

Instead, the standard condition imposed in Lorentzian geometry is the requirement of  global hyperbolicity. This implies the existence of a Cauchy surface, and an especially favourable situation is the case where the Cauchy surfaces are compact.
This is the point of view we will adopt in the following. However, the fact that the manifolds under consideration exhibit closed timelike curves in the CTC region requires  that we modify our concept of global hyperbolicity in a suitable way.
The central idea is
%to  follow the procedure employed in \cite{} \ques{If I remember correctly, a similar trick was used in other papers before. Should we cite someone and whom?} and
to consider the flat Lorentzian manifold as a surface with boundaries that are given by the CTC surfaces associated to particles. The appropriate notion of a  Cauchy surface is that of a spacelike surface with lightlike boundaries, the latter corresponding to its intersection with the CTC surfaces.

\begin{defi}
\label{def:spinGH}
A  {\em globally hyperbolic} flat Lorentzian spacetime $M$ with $n$  particles
is a flat Lorentzian manifold with particles $(M,\Delta)$ such that
$\Delta = \{ d_1, .... , d_n\}$ is the disjoint union of $n$ lines $d_1,...,d_n$
and there exist
disjoint neighbourhoods $V_1$, ... , $V_n$ of the singular lines $d_1$, ... , $d_n$ such  that:
\begin{enumerate}
\item each neighbourhood $V_i$ is isometric to
%an open subset $U_i$ of the CTC region of one singular spacetime $M_{\theta_i, \sigma_i}$.
%More precisely, we require $V_i$ to be isometric to
 a CTC cylinder of height $h_i$ in $M_{\theta_i, \sigma_i}$.

\item the complement $M^*$ of the disjoint union $\bigcup_{i=1}^nV_i$ is a flat Lorentzian manifold with boundary that admits a
Cauchy surface, i.~e.~an embedded surface with boundary $S$ with spacelike interior, such that the boundary components of $S$ are null circles in the CTC surfaces
$\partial V_i$,  and such that every inextendible causal curve in $M \setminus \Delta$ that is contained in $M^*$ intersects $S$.
\end{enumerate}
If moreover the Cauchy surface $S$ can be selected to be compact, then $M$ is called {\em spatially compact}.
\end{defi}

Note that this definition is quite restrictive regarding the CTC region around the particles.
This is due to the following reasons.  Firstly, we want the particles to be hidden behind a ``CTC surface'' $\partial V_i$, and the CTC regions $V_i$  around each particle therefore must be sufficiently
big so that they reach
the CTC surfaces in the associated one-particle models. Secondly, we need the interior region to be globally hyperbolic and hence foliated
by Cauchy surfaces. This induces a foliation of the CTC surfaces $\partial V_i\subset M$ around each particle by non-timelike closed curves, and hence
by null circles. In order to obtain a notion of globally hyperbolic flat Lorentzian  manifold with
particles that fulfils each of these requirements,   we then have to assume that each surface $\partial V_i$ is the boundary of a CTC cylinder. %This is the justification for the first item in the definition.

Given a flat Lorentzian spacetime with particles that is globally hyperbolic in the sense of Definition \ref{def:spinGH}, one can add to each CTC cylinder the
entire CTC region in the corresponding one-particle model, and this completion has no impact on the geometry of its interior part $M^*$. However, in the following, we take the viewpoint that the specific geometry of the CTC region is irrelevant itself and only of interest through its effect on geodesics that enter a connected component  $V_i$ of the  CTC region from the interior region  $M^*$ and then return to  $M^*$.  Such a geodesic has to be contained in the CTC cylinder bounded by $\partial V_i$. What happens outside the CTC cylinder inside the CTC region is
therefore not relevant to our situation except through its effects on geodesics outside the CTC region.

In the following we will focus on the situation in which the spins $\sigma_i$
are small compared to the cone angles $\theta_i$ so that the scale of the CTC radii $|r_i| = \frac{\vert \sigma_i\vert}{\theta_i}$ is small
 compared with the global geometry of the more classical globally hyperbolic interior region $M^*$. In the limit case, where one or more spins $\sigma_i$ tend to zero, the associated CTC regions $V_i$ become empty.  In that situation, one can extend the  notion
of causal curves by including curves that contain components of the singular lines.  In this setting, our notion of global hyperbolicity requires that there is a closed Cauchy surface
 intersected by all inextendible curves that are causal in that sense.

\subsection{Doubling the spacetime along the CTC surface}
\label{sub:double}

Classical results involving global hyperbolicity are not available for
spacetimes with boundary  such as the interior region $M^*$ in Definition \ref{def:spinGH}.
However, we can nevertheless relate these spacetimes to the classical framework by employing the following ``doubling the spacetime'' trick.
%\ques{more explanation here ? where it comes from, why?}

Let $M$ be a singular flat spacetime satisfying the first condition in Definition \ref{def:spinGH},  and denote by  $M^*=M\setminus\bigcup_{i=1}^n V_i$ its interior region.
For each singular line $d_i$, let $\theta_i$ be the cone angle around $d_i$, $\sigma_i$ the spin and let $r_i=\sigma_i/\theta_i$ denote the associated CTC radius.
The isometry between the boundary components $V_i$ of $M$ and the CTC cylinders defines a local coordinate system $(r,\alpha,\tau)$ in a neighbourhood $U_i\subset M^*$ of each boundary component, in which the metric takes the form \eqref{pback} (with $\theta_0$ replaced by $\theta_i$ and $\sigma$ by $\sigma_i$).
Define a new coordinate $\xi$ on $U_i$ through the condition  $r=\vert r_i \vert\cosh\xi$, $\xi\geq0$ in the interior region near each surface $\partial V_i$. In terms of these coordinates the
metric \eqref{pback} for each particle takes  the form
$$-\frac1{\theta_i^2}\,d\tau^2 - \frac{r_i}{\pi}\,d\alpha d\tau + r_i^2(\sinh\xi)^2\left[\left(\frac{\theta_i}{2\pi}\right)^2d\alpha^2 + d\xi^2\right]$$

When the CTC cylinder is a finite $h_i$-cylinder one can prescribe the coordinate $\tau$ to vary in $[-h_i/2, h_i/2]$.
As the number of particles is finite,  there exists an $\epsilon_0>0$ such that the subsets of  the neighbourhoods $U_i$
characterised by the condition $\xi\leq \epsilon_0$
define $n$ solid pairwisely disjoint cylinders
that contain all surfaces $\partial V_i$.

Consider two copies $M^*_1,$ $M^*_2$ of $M^*$ and glue them along their boundaries in the obvious way.
More precisely, let $f_1: M^* \to M^*_1$ and $f_2: M^* \to M^*_2$ be two identifications and consider the union of $M^*_1$ and $M^*_2$ with
$f_1(p)$ and $f_2(p)$ identified for every $p$ in $\partial M^*$.
We get a manifold $M^d$, containing a surface $T$
(the locus where the glueing has been performed) and two embeddings $\bar{f}_1: M^* \to M^d,$ $\bar{f}_2: M^* \to M^d$.
In the following, $M^*$ is referred to as the doubling of $M^*$ along $T$.
A neighbourhood of every connected component of $T$ in $M^d$ can be parametrized by coordinates
$(\alpha, \tau, \xi)$ but where now $\xi$ is allowed to vary in $[-\epsilon_0, \epsilon_0]$, hence to have negative values.
Positive values of $\xi$ correspond to points in the first copy $\bar{M}^*_1:=\bar{f}_1(M_1^*)$
whereas negative
values represent points in $\bar{M}^*_2:=\bar{f}_1(M_2^*)$ . The surface $\bar{T}=\bar{f}_{1,2}(T)$ is characterised by $\xi=0$.

The manifold $M^*$ is equipped with a metric $g_0$ which, however, becomes degenerate on $\bar{T}$.
Nevertheless, it is
 still reasonable to study the causality properties of such
a degenerate cone field.
A convenient way to do so is to consider $g_0$ as the limit of non-degenerate Lorentzian metrics.
For this we introduce a bump function $\eta \colon [0, +\infty) \rightarrow [0, 1]$, which is a non-increasing smooth function that
vanishes on the interval $(1,\infty)$ and takes constant value $1$ on $[0, 1/2]$. For every $0<\epsilon < \epsilon_0$, we define
\begin{align}\label{epsilonmetric}
g_\epsilon := -\frac1{\theta_i^2}\,d\tau^2 - \frac{r_i}{\pi}d\alpha d\tau + \left(r_i^2(\sinh\xi)^2 + \epsilon\eta(\vert\xi\vert/\epsilon)\right)\left[\left(\frac{\theta_i}{2\pi}\right)^2d\alpha^2 + d\xi^2\right]
\end{align}

Then $g_\epsilon$ is a Lorentzian metric on $M^d$, equal to the flat metric on the region $\xi \geq \epsilon$,
and converges with respect to the $C^0$-norm to the degenerate flat metric $g_0$ for $\epsilon \to 0$.

One can allow the coordinate $\xi$  in \eqref{epsilonmetric} to take values  on the entire real line. In this case, it defines a Lorentzian metric that becomes degenerate for $\epsilon=0$ and approximates the doubling $M^d_{\theta_i, \sigma_i}$ of the interior region $M^*_{\theta_i, \sigma_i}$ of $M_{\theta_i, \sigma_i}$.

\begin{lemma}
\label{le:epsilongh}
Every spacetime $(M^d_{\theta_i, \sigma_i}, g_\epsilon)$ for $\epsilon > 0$ is globally hyperbolic, and every level set of $\tau$ is a Cauchy surface.
\end{lemma}

\begin{proof}
The level sets of $\tau$ are spacelike for every $g_\epsilon$, hence $\tau$ is a time function.
Let $c: I \mapsto M^d_{\theta_i, \sigma_i}$ be an inextendible $g_\epsilon$-causal curve. Then it is also $g_0$-causal.
The map $\xi \mapsto \vert \xi \vert$ induces an isometric branched covering $s^d: M^d_{\theta_i, \sigma_i} \to M^*_{\theta_i, \sigma_i}$ that preserves the coordinate $\tau$.
As $M^*_{\theta_i, \sigma_i}$ is globally hyperbolic relatively to its boundary (cf. section~\ref{csurfacesec}) the image of $c$ by $s^d$ must intersect every level set
of $\tau$. The lemma follows.
\end{proof}

\subsection{A criterion for global hyperbolicity}

\begin{prop}
\label{pro:critereGH} Let $M$ be a singular flat spacetime satisfying the first condition in Definition \ref{def:spinGH}. Assume that the closure of the interior region $M^*$ contains no closed causal curves except the null circles in the  surface $\partial V_i$, and that for all $p,q\in M^*$,  the intersection
$J^+(p) \cap J^-(q)$ between the causal future of $p$ and the causal past of $q$ is either compact or empty. Then $M^*$ admits a Cauchy surface.
\end{prop}

\begin{proof}$\quad$\newline
1. This lemma  is well-known in the non-degenerate case and is at the  foundation of the notion
of global hyperbolicity.
To prove it for the degenerate case, we first observe  that the metric $g_0$ has no closed causal curves (CCC) except the null circles. Indeed,
consider the map $s \colon M^d \rightarrow M^*$ (a branched cover) that sends the points $\bar{f}_1(p)$ and $\bar{f}_2(p)$ to $p$.
It is an isometry with respect to the metric $g_0$.  If $c$ is a CCC in $M^d$ for the metric $g_0$, its image $s(c)$ under $s$ is a CCC for the flat metric in $M^*$ and  hence,
by hypothesis, a null circle. Now observe that $g_0$ weakly
dominates all the metrics $g_\epsilon$, in the sense that every causal
curve for $g_\epsilon$ is also causal for the degenerate metric $g_0$. A direct calculation shows that the null circles in the CTC surfaces are spacelike for
$g_\epsilon$. It follows that
the metrics $g_\epsilon$ have no closed causal curves.

For every point $p$ in $M^d$, denote by $J^\pm_\epsilon(p)$  the causal past (-) and future (+)  of $p$ in $M^d$ with respect to the metric $g_\epsilon$ and by $J^\pm(p)$ its causal past (-) and future (+) with respect to $g_0$. As every causal curve for $g_\epsilon$ is a causal curve for $g_0$,   the intersection $J_\epsilon^+(p) \cap J_\epsilon^-(q)$ is contained in $J^+(p) \cap J^-(q)$ for all $p,q\in M^d$. Assume that $J_\epsilon^+(p) \cap J_\epsilon^-(q)$ is not empty.
Then $s \colon M^d \rightarrow M^*$ maps $J^+(p) \cap J^-(q)$ into a closed subset of
$J^+(s(p)) \cap J^-(s(q))$. On the other hand, $s$ is a proper map. As $J^+(s(p)) \cap J^-(s(q))$
is compact by hypothesis, the same holds for
$J^+(p) \cap J^-(q)$. As  $J_\epsilon^+(p) \cap J_\epsilon^-(q)$ is a closed subset of $J^+(p) \cap J^-(q)$,
it is therefore also compact. This proves that there exists an $\epsilon_0>0$ such that the metric $g_\epsilon$ is globally hyperbolic for all  $\epsilon < \epsilon_0$.

\bigskip

2.
For every $\epsilon < \epsilon_0$ let $S_\epsilon$ be a Cauchy surface for $g_\epsilon$.
%Pick a point $p_0$ in $M^*$; we can assume wlog that the Cauchy surfaces $S_\epsilon$ contains
%$p_0$.
Denote by $\widehat{S}_\epsilon$ the intersection of the Cauchy surface $S_\epsilon$ with the interior region $M^*$, considered as a subset of $M^d$.
Denote by  $K_\epsilon$ the region in  $M^*$ that is characterised by the  condition $\xi \leq \epsilon$
and by $K^\epsilon_1$, ... , $K^\epsilon_n$ its connected components.
Recall that  $g_\epsilon$ is equal to $g_0$ outside $K_\epsilon$.
The  intersection of $\widehat{S}_\epsilon$ with every connected component $K^\epsilon_i$
is the graph of a map $(\alpha, \xi) \mapsto f_{i, \epsilon}(\alpha, \xi)$ which takes values in $(-h_i/2, h_i/2)$.

\medskip
%\textbf
{\em Claim: There is a compact $g_0$-spacelike hypersurface $S$ and a positive real number $\epsilon<\epsilon_0$ such that  $S$ coincides with a Cauchy surface $S_\epsilon$
of $(M^d, g_\epsilon)$ in the region $M^d \setminus K_\epsilon$.}

To prove the claim, we first assume that the spacetime admits only one particle ($n=1$). We fix a point $x$ in the CTC cylinder  characterised by the condition $\xi = 0$, which is the boundary of the interior region $M^\ast$. Without loss of generality, we select $x$ such that
its $\tau$-coordinate vanishes. Then, we can assume without loss of generality that the Cauchy surfaces $S_\epsilon$ have been chosen in such a way that  they all contain $x$.

By applying the Ascoli-Arzela Theorem to $f_{i,\epsilon}$, one then obtains  directly that there is a subsequence of the sequence of  surfaces $S_{1/k}$, $k\in\mathbb N$, which converges to a $g_0$-spacelike hypersurface $S_\infty$ in the region $K_{\epsilon_0} = \{ \xi\leq \epsilon_0\}$.
Note, however,  that outside  the region $K_{\epsilon_0}$, these surfaces
may escape to infinity when $k\to \infty$.  This issue can be addressed as follows: for $\epsilon$ sufficiently small, one can extend the part of $S_\epsilon$ outside $K_{\epsilon_0}$ by a surface approximating $S_\infty$, which is $g_0$-spacelike (details are left to the reader).%\ques{Are you sure this works? Maybe we should give some details}
We then obtain a compact surface $S$ which, as required, is $g_0$-spacelike and coincides with $S_\epsilon$ outside of $K_{\epsilon_0}$ (recall that $g_\epsilon$ and $g_0$ coincide there). This proves the claim for $n=1$.

Consider now the case $n\geq 2$. Fix a point $x_1$ in the CTC cylinder in the first component $K^{\epsilon_0}_1$, and assume that every $S_\epsilon$ contains $x_1$.
Reasoning as above, we construct a surface $\Sigma_1$ which coincides with $S_{\epsilon_1}$ (for some $\epsilon_1$) outside $K^{\epsilon_0}_1$ and is
 $g_0$-spacelike in $K^{\epsilon_0}_1$.
Denote by $\Sigma_1^+$, $\Sigma_1^-$ the two compact surfaces obtained by pushing in the future (respectively in the past) the surface $\Sigma_1$ in such a way that the resulting surfaces are $g_0$-spacelike in $K^{\epsilon_0}_1$
and $g_{\epsilon_1}$-spacelike outside $K^{\epsilon_0}_1$.
We consider the region $W_1$ between $\Sigma_1^+$ and $\Sigma_1^-$. As the surfaces $\Sigma^\pm_1$ are $g_\epsilon$-spacelike for every $\epsilon$, $W_1\setminus (K^{\epsilon_0}_2 \cup \ldots \cup K^{\epsilon_0}_n)$
is globally hyperbolic for every $g_\epsilon$. It follows that the surfaces $S_\epsilon$ can be  selected in such a way  that they all lie in $W_1$, with the possible exception of the region $K^{\epsilon_0}_2 \cup \ldots \cup K^{\epsilon_0}_n$.

We now drop the condition $x_1 \in S_\epsilon$ and replace it by an analogous condition for the second connected component:  we impose that all surfaces $S_\epsilon$ contain a given point $x_2$ in the CTC cylinder in $K^{\epsilon_0}_2$.  Repeating the argument above, we obtain two disjoint surfaces $\Sigma^+_2$, $\Sigma_2^-$ which
\begin{itemize}
\item are chosen in such a way that $\Sigma_2^+$ lies in the future of $\Sigma_2^-$
\item are $g_0$-spacelike in the region $K^{\epsilon_0}_1 \cup K^{\epsilon_0}_2$,
\item are $g_{\epsilon_2}$-spacelike in $K^{\epsilon_0}_3 \cup \ldots \cup K^{\epsilon_0}_n$,
\item lie between the surfaces $\Sigma_1^+$ and $\Sigma_1^-$ outside of $K^{\epsilon_0}_2 \cup \ldots \cup K^{\epsilon_0}_n$.
\end{itemize}
We now impose as a condition that the Cauchy surfaces $g_\epsilon$ lies between $\Sigma_2^-$ and $\Sigma_2^+$, with the possible exception of region $K^{\epsilon_0}_3 \cup \ldots \cup K^{\epsilon_0}_n$.  Iterating this process, we obtain two compact surfaces $\Sigma_n^+$, $\Sigma_n^-$ which are $g_0$-spacelike everywhere. We conclude the proof of the claim as in the case $n=1$.

\bigskip
3. After proving the claim, we  resume our proof of Proposition~\ref{pro:critereGH}.
To conclude this proof, we show that the surface $\widehat{S} = S \cap M^\ast$ is a Cauchy surface for $g_0$. Let $\gamma: I\rightarrow M^*$, $t \mapsto \gamma(t)$ be an inextendible future oriented causal curve in $M^*$
for the metric $g_0$. Assume without loss of generality that $\gamma(0)$ lies in the past of $\widehat{S}_\epsilon$ for the metric $g_\epsilon$. By way of
contradiction, assume that $\gamma$ never intersects $\widehat{S}$.

Define   $t_0=\text{sup}\{t\in I\,|\, \gamma(s)\notin K_0\,\forall s\in I, s\leq t\}$ . By definition of $K_0$, this implies that  for all  $t \leq t_0$ $\gamma(t)$ lies in the region where $g_0$ and $g_\epsilon$ are equal.  Hence the restriction of
 $\gamma$ to the interval ${I\cap (-\infty, t_0]}$ is a causal curve with respect to $g_\epsilon$. As the surface $\widehat{S}_\epsilon$ is a Cauchy surface for $g_\epsilon$ and
 coincides with $\widehat{S}$ outside $K_0$, it follows from the assumption  that $t_0$ must be finite. Moreover,   $\gamma(t_0)$
lies in the past of the Cauchy surface  $\widehat{S}_\epsilon$ and therefore under the graph of $\hat{f}$.

However, by hypothesis,
$K_0\cap \gamma(I)$,  cannot intersect the graph of $\hat{f}$. It follows that $\gamma$ must exit the region $K_0$ and that its exit point
it is still in the past of $\widehat{S}_\epsilon$ with respect to $g_\epsilon$.
Let $T=\sup\{t\in I:\gamma(t)\in K_0\cap J^-_\epsilon(\widehat{S}_\epsilon)\}$ and
let $(t_n)_{(n \in \N)}$ be an increasing sequence such that $t_n \to T$. Observe that $T$ might be infinite.
Every  point $\gamma(t_n)$ lies in the future of $\gamma(0)$.
As $g_\epsilon$ is globally hyperbolic, $J^+_\epsilon(\gamma(0)) \cap J^-_\epsilon(\widehat{S}_\epsilon)$ is compact and  the sequence $\gamma(t_n)$ converges.  As $\gamma$ is inextendible, it follows that $T$ is finite, and the  limit must be $\lim_{n\to\infty} \gamma(t_n)=\gamma(T)$. In particular, this implies that $T$ is finite and  $\gamma(T)$ lies
in $K_0 \cap J^-_\epsilon(\widehat{S}_\epsilon)$. By hypothesis, $\gamma(T)$ is not in $\widehat{S}_\epsilon$ because $\widehat{S}_\epsilon \cap K_0 \subset S$. This implies that for some $t > T$,
$\gamma(t)$ is still in the past of $\widehat{S}_\epsilon$. But the argument above, proving that $t_0$ is finite, implies that $\gamma$ should meet $K_0$ once more, which is a
contradiction.

This implies  that every  inextensible causal curve  with respect to $g_0$  must intersect the surface $\widehat{S}$. Hence $\widehat{S}$ is a Cauchy surface, and $M^*$ is globally hyperbolic.
\end{proof}

%\begin{remark}
%\label{rk:cauchynotsmooth}
%In order to avoid certain technical difficulties, not directly relevant to our main concern,
%we do not require the Cauchy surface to be smooth, but just to
%be a closed achronal subset of $M^*$. It is a ``folkloric'' fact that the existence of a Cauchy surface with low regularity
%implies the existence of smooth ones. But the literature on the topic is notoriously questionable, full of
%mistakes, and we are not aware
%of a valid suitable smoothability result in the context of spacetimes with boundaries as considered here. See \cite{sanchez},
%and also \cite{fathi}.

%\end{remark}

\section{Construction of stationary flat spacetimes} \label{statsec}

\subsection{Euclidean surfaces with cone singularities}\label{euclcone}

%\subsection{Definition}

After discussing the one-particle model and introducing a notion of global hyperbolicity, we will now construct examples of flat  Lorentzian spacetimes with particles. The resulting spacetimes are stationary and the construction is based on Euclidean surfaces with cone singularities.
In the following, we denote by    $\R^2_\theta$ for $\theta>0$ the  Euclidean plane with one singular point of cone angle $\theta$, that is $\R^2 \setminus \{ 0 \}$ with the metric given in
\eqref{planemetric}.

%Let $\Sigma$ be a closed orientable surface.
\begin{defi}
A Euclidean metric with cone singularities on a closed orientable surface $\Sigma$ consists of a finite number
of points $p_1$, ... , $p_n$ (the cone singularities) together with an assignment of  positive real numbers $\theta_i>0$ (the cone angles) to $p_i$ for $i=1,...,n$,
and a flat Riemannian metric on $\Sigma^* = \Sigma \setminus \{ p_1, ... , p_n \}$, such that every point $p_i$
admits a neighbourhood $U_i$ in $\Sigma$ so that $U_i \setminus \{ p_i \}$
is isometric to a ball in $\R^2_{\theta_i}$ centred at the singular point.
\end{defi}

Note that the quantities $\mu_i := 2\pi - \theta_i,$ which in (2+1)-gravity are interpreted as masses  of the particles, are usually  referred to as  \textit{apex curvatures} in the mathematics literature (see for example \cite{thurstonshape}).   They are subject to the relation
$$2\pi\chi(\Sigma) = \sum_{i=1}^n \mu_i$$
where $\chi(\Sigma)$ denotes the Euler characteristic of $\Sigma$. In particular,
if all the cone angles satisfy $\theta_i\leq 2\pi$, then the surface $\Sigma$ is either the sphere of the torus, and the torus arises only if there is no singularity.

Observe that the flat Euclidean structure  naturally defines a conformal structure on $\Sigma^*$ and  the punctures $p_i$ correspond to the cusps of this conformal structure.
Consequently, the flat Euclidean structure  equips $\Sigma$ with the structure of a Riemann surface. That the converse is also true follows from a theorem by Troyanov.

\begin{theorem}[\cite{troyanov1}]
Let $p_1$, ... , $p_n$ be a collection of $n$ points in $\Sigma$, and $\theta_1$, ... , $\theta_n$ positive real numbers such that
$$2\pi\chi(\Sigma) = \sum_{i=1}^n (2\pi-\theta_i)$$
Then, for any conformal structure on $\Sigma$, there is an Euclidean metric on $S$ with cone singularities
of cone angles $\theta_i$ at each $p_i$ that induces the given conformal structure. This
singular Euclidean metric is unique up to a rescaling factor - in particular, it is unique
if we require the total volume to be equal to $1$.
\end{theorem}

The study of singular Euclidean surfaces is a very traditional topic in mathematics.
For instance, it is related to billiards.  A way of investigating a billiard in a polygon in the Euclidean plane is
to consider it as the geodesic flow of the singular flat Euclidean metric on  the sphere, which is
obtained by taking the double of the polygon along its sides (see \cite{matab}).

An important  case is the  one in which all  cone angles are  rational.
For instance, if all of these angles are multiples of $\pi$, the associated singular Euclidean metric is directly related to holomorphic quadratic differentials. %This observation is at the source  of several important results and has triggered very active research.
%\ques{Can we be more specific which ones?}

On the other hand,  the ``positive curvature case'', in which all  cone angles are less than $2\pi$ and  the Euclidean
surface is a sphere sphere with conical singularities,
there is always a geodesic triangulation of the singular Euclidean surface $\Sigma$. This implies that $\Sigma$  can be obtained by gluing triangles in the Euclidean plane along their sides (see~\cite[Proposition~3.1]{thurstonshape}).
In particular, when all the cone angles are rational, the associated flat surface is an orbifold. It is obtained as a quotient of a closed
Euclidean surface without cone singularities by the action of a finite group of isometries.

\subsection{Stationary flat spacetimes with particles} We now construct  globally hyperbolic flat spacetimes with particles based on Euclidean surfaces with cone singularities.
The simplest and most  obvious example are static spacetimes, which are obtained as a direct product of the Euclidean surface with cone singularities with $\R$.

\begin{defi}[Static spacetimes with particles] \label{statspt}
Let $\Sigma$ be a closed Euclidean  surface with conical singularities $p_1$, ... $p_n$ of angles $\theta_1$, ..., $\theta_n$.  We denote by $ds_0^2$ the flat metric on the regular part $\Sigma^*$. Then the  product $M=\Sigma \times \R$ contains the
open subset $M^* = \Sigma^* \times \R$ and can be equipped with the Lorentzian metric $ds_0^2 - dt^2$, where the coordinate $t$ parametrises $\R$. This metric is locally flat,
and can be considered as the regular part of a flat singular metric on $M$ where the lines $\{ p_i \} \times \R$ are
spinless particles of cone angle $\theta_i$.
\end{defi}
Observe that these spacetimes are static: the vertical vector field $\partial_t$   is a Killing vector field, orthogonal to the level
sets of $t$. As the spacetime is static,
 $t$ is  a Cauchy time function and  the levels of $t$ are compact and hence complete.
This implies directly that  the static singular flat spacetime $M$ is globally hyperbolic.

%\subsection{Stationary flat spacetimes}

To obtain a more interesting class of examples, we consider
 a closed $1$-form $\omega$ on $\Sigma^*$, where $\Sigma^*$ is the regular part of a singular flat Euclidean surface as in Definition \ref{statspt}. We consider again the direct product $\Sigma^*\times \R$ but now equipped with the metric \begin{align}\label{ommetric}ds_\omega = ds_0^2 - (\omega + dt)^2
 = ds_0^2 - \omega^2 -dt^2 - 2\omega dt\end{align}
 instead of $ds_0^2 - dt^2$.
 Note that this defines a flat Lorentzian metric
 on  $M^*=\Sigma^*\times \R$.  On any
 subset  $U \times \R$ where $U\subset \Sigma^*$ is contractible, the form $\omega$ is exact, i.~e. ~ given as the differential $\omega = df$ of a map $f \colon U \to \R$.
This implies that the metric $ds_\omega$ on $V\times\R$  is simply the pull-back of $ds_0^2$ under the diffeomorphism $(x,t) \mapsto (x, t + f(x))$ and hence a flat Lorentzian metric on $V\times\R$. Moreover,
this argument shows that $ds^2_\omega$ only depends, up to isometry, on the cohomology
class of $\omega$.

We now fix open pairwisely disjoint neighbourhoods $U_i\subset \Sigma$ around every singular point $p_i$  such that $U^*_i:=U_i \setminus \{ p_i \}$
is isometric to a ball in $\R^2_{\theta_i}$ centred at the singular point. In a suitable polar coordinate
system, the metric on $U_i$ then takes the form
$$ds_i^2=dr^2 + \left(\frac{\theta_i}{2\pi}\right)^2r^2d\alpha^2.$$

Denote by $\omega_i$ the closed $1$-form $d\alpha/2\pi$ in $U^*_i$. As it generates the first cohomology group,
the restriction of $\omega$ to $U^*_i$ is cohomologous to $\sigma_i\omega_i$, where $\sigma_i =\int_{\gamma_i} \omega$ and  $\gamma_i$
 is a loop in $U^*_i$ that makes one positive turn around $p_i$. Therefore, there exists a function $f_i \colon U^*_i \rightarrow \R$
such that $\omega = \sigma_i\omega_i +df_i$.
Let $f:\Sigma^*\rightarrow \R$ be a function whose restriction to $U_i^*$ coincides with $f_i$ for all $i\in\{1,...,n\}$. Such a function can be constructed  by means of bump functions $b_i: \sigma^*\rightarrow \R$, which satisfy $b_i|_{U_i^*}=1$ and $b_i|_{U_j^*}=0$ for all $i\neq j$. The function is then given by
$
f=\sum_{i=1}^n b_i f_i
$.
On every neighbourhood $U_i^*$, the one form $\omega'=\omega-df$ then
 coincides with $\sigma_i\omega_i$. This implies that the associated metric $ds_\omega^2$ defined as in \eqref{ommetric} and restricted to
  $V_i^*:=U^*_i \times \R$ takes the form
$$ds_\omega^2\vert_{V_i}=dr^2 + \left(\frac{\theta_i}{2\pi}\right)^2r^2\,d\alpha^2 - \left(\frac{\sigma_i}{2\pi}\right)^2\!d\alpha^2 - \frac 1 \pi \,d\alpha dt-dt^2$$

We recognise, up to a rescaling factor $1/\theta_i$ for the coordinate $t$, the metric \eqref{metric2} for a particle with spin $\sigma_i$.
Hence $(M, ds^2_\omega)$ is the regular part of a singular flat Lorentzian metric with particles on $\Sigma \times \R$ according to Definition \ref{mfdef}. We are therefore free to make the following definition.

\begin{defi} [Stationary spacetimes with particles]\label{omspt}

Let $\Sigma$ be a closed Euclidean  surface with conical singularities $p_1$, ... $p_n$ of angles $\theta_1$, ..., $\theta_n$ and  $ds_0^2$ the flat metric on its regular part $\Sigma^*$.
Let $\omega$ be a closed $1$-form on $\Sigma^*$. Then the flat stationary singular spacetime $(M,\omega)$ associated with $M=\Sigma\times\R$ and $\omega$ is the product $M^*=\Sigma^*\times \R$ equipped with the flat Lorentzian metric
\eqref{ommetric}.
\end{defi}

Note that we did not state that the stationary spacetimes defined in this way are globally hyperbolic
 in the sense of Definition \ref{def:spinGH}, as this is  not the case in general .
We will give a necessary and sufficient set of conditions for global hyperbolicity in Section \ref{propstat}.

\subsection{Geometrical properties of stationary flat singular spacetimes}\label{propstat}

 %\subsection{Geometrical interpretation of the closed $1$-form}\label{sec:omegaconnection}

% We are now ready to investigate the  properties of stationary flat spacetimes. We start by discussing the geometrical interpretation of the closed $1$-form $\omega$.

 \subsubsection{Geometrical interpretation of the closed $1$-form}\label{sec:omegaconnection}
 Let $(M, ds^2_\omega)$ a singular flat spacetime as in Definition \ref{omspt}
defined by  a
Euclidean closed surface  $\Sigma$ with cone singularities and a closed $1$-form $\omega$ on the regular part $\Sigma^*$.

Then it is immediate from \eqref{ommetric} that the translations
along the $t$-coordinate  are isometries and the vertical vector field ${\partial_t}$ is a Killing vector field, which is timelike everywhere. The space of trajectories of this vector field is the space of vertical lines; it is naturally identified
with $\Sigma$. Actually, the projection on the first factor is a (trivial) $\R$-principal fibration $\nu \colon M \rightarrow \Sigma$.
The orthogonal complements of $\partial_t$ for $ds^2_\omega$ define a plane field transverse to this fibration restricted over $\Sigma^*$, i.~e.~a connection on this restricted $\R$-bundle.
More precisely, $\omega$ is the $1$-form relating this connection to the trivial product connection
on $\Sigma^* \times \R$ which arises  from the static metric $ds_0^2-dt^2$.  As the curvature $d\omega + \omega\wedge\omega$ vanishes, the connection associated with $ds^2_\omega$ is flat. This can also be deduced in a more elementary way  from the fact that horizontal planes characterised by the condition $t=$ constant are orthogonal to the Killing vector field ${\partial_t}$.

That the metric  $ds^2_\omega$ only depends on  the cohomology class of $\omega$
reflects the fact that the trivialisation of the $\R$-principal fibration
$\nu$ is unique only up to
 gauge transformations. The latter are  translations in the fibers and hence determined by a function $f \colon \Sigma' \rightarrow \R$. They result in a change of $\omega$ by a coboundary $\omega\mapsto \omega+df$.

%\begin{remark}
As a direct application of Mayer-Vietoris sequence, one finds that a tuple of real numbers  $(\sigma_1, ... , \sigma_n)$ can be realised as the residues  of a closed $1$-form around the cone singularities $p_1,...,p_n$ on the surface $\Sigma$ if and only if the sum $\sigma_1 + ... + \sigma_n$ vanishes.  Moreover, once the residues $\sigma_1,...,\sigma_n$ are prescribed, the $1$-form $\omega$ is unique up to a closed $1$-form on the closed surface $\Sigma$. In particular, when $\Sigma$ is a sphere , then the residues $\sigma_1,...,\sigma_n$
determine the cohomology class $[\omega]$.

In the application to (2+1)-gravity, the residues $\sigma_1,...,\sigma_n$ correspond to the spins of $n$ massive particles associated with the singularities. If each particle has positive mass, then it follows from the discussion in Section \ref{euclcone} that  the resulting Euclidean surface $\Sigma$ is a sphere. In that case, the associated stationary spacetime is then determined uniquely  by the spins of the particles which must satisfy the condition $\sigma_1+\ldots+\sigma_n=0$.

%\end{remark}

 \subsubsection{Geodesics and Completeness}
  If the closed $1$-form  is exact, $\omega=df$,  which is always true locally,  the associated metric $ds^2_\omega$
is given as the pull-back of $ds_0^2$ under the diffeomorphism $(x,t) \mapsto (x, t + f(x))$. This allows one to give the following simple description of the geodesics in
 $(M, ds^2_\omega)$. A geodesic $g$ in  $(M, ds^2_\omega)$
is a path $$g: t \mapsto \left(c(t)\;,\; t_0+\lambda\int_{0}^t \omega(\dot{c}(s))ds\right),$$
where $\lambda\in\R$ and $t \mapsto c(t)$ is a geodesic in $\Sigma$ parametrised by arc length.
The geodesic $g$ is timelike if $\vert\lambda\vert > 1$, lightlike if $\vert\lambda\vert = 1$,
and spacelike if $\vert\lambda\vert < 1$.  Note that our notion of geodesic is not the one of a curve that minimises a length functional.
 In particular, we do not exclude that
the  geodesic $c$ in $\Sigma$  contains the singular points, which implies that  the geodesic $g$ can go through the singular lines.  The singular lines themselves can be considered as geodesics according to this definition.

Using this notion of geodesics, we obtain a natural definition of geodesic  completeness that allows us to directly deduce that the stationary spacetimes in Definition \ref{omspt} are complete.
We call a stationary spacetime $(M, ds^2_\omega)$  geodesically complete if inextensible geodesics are defined on the entire real line $\R$. This leads to the following proposition.
%; $(M, ds^2_\omega)$ is future (respectively past) complete if this
%property holds only for future oriented (respectively past oriented) causal geodesics, i.e.. in the case $\lambda\geq1$ (respectively $\lambda\leq-1$).

\begin{prop}
The spacetime $(M, ds^2_\omega)$  is geodesically complete.
\end{prop}

\begin{proof} As our notion of geodesics allows them to traverse the singularities,
the underlying Euclidean surface $(\Sigma, ds^2_0)$ is geodesically complete, since it is compact.  As the geodesics of $(M, ds^2_\omega)$ take the form
 $g: t \mapsto (c(t), t_0+\lambda\int_{0}^t \omega(\dot{c}(s))ds)$
where $\lambda\in\R$ and $t \mapsto c(t)$ is a geodesic in $\Sigma$ parametrised by arc length, this also holds for the geodesics in $(M, ds^2_\omega)$.
%Observe
%that this does not cause any problems near the singular points, as we have shown that that the geometry of their neighbourhoods is given by
%a neighbourhood of the singular line in $M_{\theta_0, \sigma}$.\ques{This should be made more precise - what problems could arise? How is this related with the first part of the argument?}
\end{proof}

We are now read to investigate under which conditions the stationary flat spacetimes with particles are globally hyperbolic.
The answer to this question is provided by the following proposition.

\begin{prop}
\label{pro:ghspin}
The flat stationary singular spacetime $(M, ds^2_\omega)$ is globally hyperbolic if and only if  the following three conditions are satisfied:
\begin{enumerate}

\item For every singular point $p_i$, the Euclidean ball centred at  $p_i$ and of  CTC radius
$|r_i|= \frac{\vert \sigma_i \vert}{\theta_i}$ is embedded, i.~e.~the injectivity radius at $p_i$ for the singular Euclidean metric
is greater than $r_i$.

\item For every pair $(p_i, p_j)$ of singular points the sum of their CTC radii is greater than their Euclidean distance  $|r_i| + |r_j| < d(p_i, p_j)$.
%where $d(p_i, p_j)$ is the Euclidean distance of  $p_i$ and $p_j$.

\item Let $S_0$ be the set of points  $p\in \Sigma$ for which the Euclidean distance $d(p,p_i)$ from each singular point $p_i$
is strictly greater than the CTC radius $|r_i|$.
Then the absolute value of the integral of $\omega$ along any  closed loop $\gamma$ in $S_0$ is strictly smaller than the Euclidean length $\ell_0(\gamma)$ of $\gamma$
$$\left\vert \int_{\gamma} \omega \right\vert < \ell_0(\gamma).$$
\end{enumerate}
\end{prop}

The first two conditions are immediately recognised as  equivalent to the first condition in Definition \ref{def:spinGH}, which states that the CTC regions associated to particles are disjoint CTC cylinders. The third condition implies that
 if the residues of $\omega$
are sufficiently small, then $(M, ds^2_\omega)$ is globally hyperbolic in the sense of Definition \ref{def:spinGH}.
Note that if
 $\gamma_i$ is the circle of radius $|r_i|$ and centre $p_i$, we have the equality
$$\left \vert \int_{\gamma_i} \omega \right \vert = \vert \sigma_i \vert = \theta_i|r_i| = \ell_0(\gamma_i).$$
The inequality in the third condition must therefore become an equality at the boundary of $S_0$.

\begin{proof}

\textit{The conditions are necessary:} Assume that $(M, ds^2_\omega)$ is globally hyperbolic in the sense of Definition \ref{def:spinGH}. Then the CTC regions around the
particles must be pairwise disjoint, and every particle must be contained in a CTC cylinder. It is clear that these properties imply
items $(1)$ and $(2)$.
Moreover, by definition of global hyperbolicity,  the interior region of $M$ must contain a Cauchy surface $S$, which intersects every vertical line in exactly one point.
This implies that the Cauchy surface $S$ is the
 graph of a function $f: \overline{S}_0 \rightarrow \R$,
where $\overline{S}_0$ is the closure of $S_0$.
As this graph is spacelike, we have for every non-vanishing tangent vector $v$ of $S_0$
$$0 < ds^2_0(v) - (\omega(v) + df(v))^2$$
and therefore
$$ - ds_0(v) < \omega(v) + df(v) < ds_0(v).$$
The inequality in the condition $(3)$ is  then obtained directly by integration along any closed loop $\gamma$ in $S_0$, since the integral of $df$ along
such a loop vanishes.

\bigskip
\textit{The conditions are sufficient:} As  already observed, the first two conditions equivalent to the statement  that every particle is surrounded by a CTC
cylinder, and that these cylinders are disjoint.  The inequality in condition $(3)$ means precisely that there are no closed causal curves (CCCs)  in $S_0 \times \RR$, hence the only CCC in $M^*$ are the null circles in the CTC surface. It remains to construct a Cauchy surface for $M^*$.

Let now $q=(x,t_0)$ be a point in the closure $\overline{M^*}$ of $M^*$. For every $y$ in $\overline{S}_0$, select a path $c \colon [a,b] \rightarrow \overline{S}_0$
such that $c(a)=x$ and $c(b)=y$. Then the curve $t \mapsto (c(t), t_0+\int_{[a,t]} ds_0(\dot{c}(s)) - \omega(\dot{c}(s)) ds )$ is a future oriented
lightlike curve in $\overline{M^*}$ joining $q$ to a point $q(y)$ in the vertical line above $y$. Hence, there exists a $t\in\R$ such that $(y,t)$ lies
in the causal future $J^+(q)$ of $q$. Similarly, there exists a $t'\in\R$ such that $(y, t')$ lies in $J^-(q)$.

As there is no
CCC above $S_0$, we must have $t' < t$, except if $x$, $y$ lies in the same boundary component of $S_0$, in which case we must have $t=t'$. (Otherwise
we could construct a CTC in $M^*$).
In particular, there is an upper bound for $t'$, and a lower bound for $t$. In other words,
for any $q\in \overline{M^*}$ there are two maps $f^\pm_q \colon \overline{S}_0 \rightarrow \R$ such that  a point $p=(y,t)$ in $S_0 \times \R$
lies in $J^+(q)$ (respectively $J^-(q)$)
if and only if $t \geq f^+_p(y)$ (respectively $t \leq f^-_q(y)$).

Observe that the future $I^+(q)$ of $q$ in $\overline{M^*}$ is the set of points $(y, t)$ with $t > f^+_q(y)$: by adding small $t$-components to the lightlike curves
considered above one can obtain lightlike curves joining $q$ to $(y, t)$ for any $t > f^+_q(p)$. As $I^+(p)$ is open,
it follows that $f^+_q$ is upper semi-continuous. Similarly, $f^-_q$ is lower semi-continuous. As $ \overline{S}_0$ is compact,
the set of points $(x,t)$ such that $f^-_p(x) \leq t \leq f^+_q(x)$ is compact. But this set is precisely the intersection $J^-(q) \cap J^+(p)$.
Proposition~\eqref{pro:critereGH} then implies the existence of a Cauchy surface and hence the global hyperbolicity of $M$.
\end{proof}

In addition to establishing criteria for the global hyperbolicity of the stationary singular spacetime $(M, ds^2_\omega)$ this proposition provides a characterisation of the $1$-forms associated with
globally hyperbolic spacetimes. We obtain the following corollary.

\begin{cor}
\label{cor:omega1}
Let $(M, ds^2_\omega)$ be the singular flat spacetime associated to a closed singular Euclidean surface $\Sigma$ and a
closed $1$-form $\omega$ on $\Sigma^*$. If $(M, ds^2_\omega)$ is globally hyperbolic, then $\omega$ is cohomologous to a $1$-form
$\omega_0$ such that
%for every path $\gamma \colon [a,b] \rightarrow S_0$ (not necessarily a loop!) we have:
%$$ \int_\gamma \omega_0 < \ell_0(\gamma)$$
at every point $x$ of $S_0$, the operator norm of $\omega_0,$ as a linear form on $T_xS_0$, and with respect to the norm $ds_0$ is  less than one.
\end{cor}

\begin{proof}
Consider a smooth map $f_0 \colon \overline{S}_0 \rightarrow \RR$ whose graph is a Cauchy surface in $M^* = \overline{S}_0 \times \R$,
and take $\omega_0 = \omega + df_0$. The fact that the graph of $f_0$ is spacelike is equivalent to the condition  that for any path $\gamma \colon [a,b] \rightarrow S_0$ (not necessarily a loop!)
the following inequality holds$$ \int_\gamma \omega_0< \ell_0(\gamma)$$
As this inequality holds for any path $\gamma:[a,b]\rightarrow S_0$, we obtain an infinitesimal version by differentiation: the evaluation of $\omega$ on any
tangent vector $v$ is bounded by the $ds_0$-norm of $v$. The corollary follows.
\end{proof}

%REMARQUE: EFFET DE MODIFICATION PAR LA MASSE DE LA COORDONNEE t !!!!

\subsection{Classification of stationary flat singular spacetimes}
We already observed that the spacetimes $(M, ds^2_\omega)$ are stationary. We will now show that there is a  converse statement that allows one to relate any stationary flat spacetime with a compact spatial surface and  particles to a spacetime as in Definition \ref{omspt}.

\begin{prop}
Let $M$ be a flat globally hyperbolic spatially compact spacetime with particles. Assume that $M$ is stationary,
i.~e.~that the regular region $M^*$ admits a Killing vector field $X$ which is timelike everywhere. Then $M$ embeds isometrically in the flat singular
spacetime $(M, ds^2_\omega)$ associated to a singular Euclidean surface $\Sigma$ and a closed $1$-form $\omega$ on the regular
part $\Sigma^*$ of $\Sigma$.
\end{prop}

\begin{proof}
Denote by $M^*$ the regular part of $M$ given as the complement of the singular lines, and by $M^*_0$
the interior region that is the complement of the CTC regions. Let $S$ be a compact Cauchy surface in $M^*_0$.
We denote by $\Mw$, $\Mw^*$, $\Mw^*_0$, $\Sw$, respectively,  the universal coverings of these manifolds.
Then there are natural inclusions $\Sw \subset \Mw^*_0 \subset \Mw^*$.

By definition, the regular part $M^*$ is locally modelled on  Minkowski space $\R^{1,2}$.
Let $\cD \colon \Mw^* \rightarrow \R^{1,2}$ be the associated developing map, and let
$\rho \colon \Gamma \rightarrow \op{SO}_0(1,2) \ltimes \R^{1,2}$ be the holonomy representation,
where $\Gamma$ is the fundamental group of $M^*$.
The Killing vector field $X$ induces a local flow $\phi^t$ on $M^*$. This flow is isometric and maps  closed timelike curves to closed timelike curves.
It  follows that
$\phi^t$ preserves every CTC region $U_i$ around the singular lines $d_i$. Let $\tilde{\phi}^t$ be the lift of
of $\phi^t$ to $\Mw^*$, which is  generated by the lift $\tilde{X}$ of $X$. Then the orbit space of $\tilde{\phi}^t$ is naturally identified with $\Sw$.

It is well-known that local isometries on open subsets of $\R^{1,2}$ extend to isometries of  $\R^{1,2}$. This implies that the lift $\tilde{X}$ is the pull-back $\cD_*X_0$ of a Killing
vector field $X_0$ on $\R^{1,2}$.
Let $\widetilde{U}_1$ be a lift of  the region $U_1$ around $\tilde{d}_1$ to $\Mw^*$.
Then there exists an isometry of $\R^{2,1}$ whose composition with the developing map $\cD$ maps $\widetilde{U}_1$
into cylinder $\{ (r,\theta, t)\,: 0\leq \theta\leq 2\pi, r^2 \leq r_1^2, a < t < b \}$ in Minkowski space. Note that we admit also cylinders with  $a$ and/or $b$ are infinite. It follows that the flow associated with $X_0$  must preserve this cylinder and, consequently,  $X_0$ is a linear combination $\alpha {\partial_\theta} + \beta{\partial_t}$.
%\ques{what I do not understand here is why one looks only at $U_1$ and not at the other $U_i$s. Because the reasoning is similar? then one should say so. Or what is the reason?}

Suppose $\alpha \neq 0$. Then, the only  timelike line invariant under the flow of $X_0$ is the vertical line $\Delta_0$ characterised by the condition $r=0$. But for every
$\gamma$ in $\Gamma$, $\rho(\gamma)(\Delta_0)$ is also a timelike $X_0$-invariant line: $\Delta_0$ is therefore
preserved by the entire holonomy group $\rho(\Gamma)$, which therefore is contained in the subgroup of
rotations with axis $\Delta_0$. In particular, the holonomy preserves the vertical vector field ${\partial_t}$.
The pull-back $\cD^*({\partial_t})$ is a Killing vector field, which is timelike everywhere. Hence
we can  replace $X_0$ by this Killing vector field, since it satisfies the same hypothesis.

We can therefore assume $\alpha=0$ and  identify $X_0$ with the pull-back $\cD^*({\partial_t})$ up to an homothety.
This implies that every element of $\rho(\Gamma)$ preserves the vertical lines in $\R^{1,2}$. Elements of $\rho(\Gamma)$  must therefore be given as the  composition of a rotation
around a vertical axis and a translation.
Let $\R^2$ be the horizontal plane characterised by the condition $t=0$ and let  $\pi_0 \colon \R^{1,2} \rightarrow \R^2$ be the
orthogonal projection. This projection is equivariant under the action of  $\Gamma$ on $\R^{1,2}$
through the holonomy representation $\rho$ and its isometric action on $\R^2$. More precisely, for every $\gamma\in\Gamma$, we have $\pi \circ \rho(\gamma) = \bar{\rho}(\gamma) \circ \pi$ where $\bar{\rho}(\gamma)$ is the isometry
of $\R^{1,2}$ which has the same linear part as $\rho(\gamma)$ (some rotation) and whose translation part is the
projection of the translation part of $\rho(\gamma)$.

Then, the restriction of $\overline{\cD} = \pi_0 \circ \cD$ to $\widetilde{S}$ and $\bar{\rho} \colon \Gamma \rightarrow \op{Iso}(\R^2)$ are
the developing map and holonomy representation of a Euclidean structure on $S$.
By investigating the behaviour of of this maps in a neighbourhood of each singular line $d_i$, one finds that this Euclidean
structure extends to a singular Euclidean metric on a closed surface $\Sigma$ that is obtained from $S$ by adding to the boundary of $S$ round disks with one cone-angle singularity.

The maximal trajectories of $\phi^t$ are
timelike inextendible curves in $M^*$.
 As it is a fiber of a fibration $\nu \colon M_0^* \rightarrow S$, every trajectory of $\phi^t$ in the interior region $M^*_0$ intersects $S$ in exactly one point.
  By  considering this fibration in every CTC region, one finds that it
extends naturally on the regular region $M^*$ to a fibration $\nu \colon M^* \rightarrow \Sigma^*$, whose fibers are trajectories
of $\phi^t$. In particular, the fibers are homeomorphic to $\R$: there is a global section $\sigma \colon \Sigma^*\rightarrow M^*$.
For every $p$ in $M'_0$, let $\Psi(x)$ be the pair $(\nu(p), t(p))$, where $t(p)$ is the unique real number $t$ such that $p = \phi^t(\sigma(\nu(x))$:
it defines an embedding $\Psi \colon M'_0 \rightarrow \Sigma' \times \R$ such that $\nu = p_1 \circ \Psi$, where $p_1$ is the projection
on the first factor.

The orthogonal complement to $X_0$ with respect to the flat metric defines a plane field on $M^*$ which is transverse to $\nu$ and invariant under
vertical translations. It therefore extends to a connection on the trivial  $\R$-principal  bundle $p_2 \colon \Sigma^*\times \R \rightarrow \Sigma^*$, which is
invariant under $\R$-translations on the fiber. This connection is flat, since in $\R^{1,2}$ the plane field orthogonal to $\frac\partial{\partial t}$
is integrable.

Let $\omega$ be the $1$-form relating this connection to the trivial connection on $\Sigma^* \times \R$. As both connections are flat, it is
 a closed $1$-form. This implies that the bundle embedding $\Psi \colon M^* \rightarrow \Sigma^* \times \R$ maps
the plane field orthogonal to $X_0$ on the plane field in $\Sigma^*\times \R$ which is orthogonal to the fibers for the metric $ds_\omega^2$
(cf. section~\eqref{sec:omegaconnection}). It follows quite easily that $\Psi$ is an isometric embedding.%\ques{How does this follow? explain}

%Finally, for every $p$ in $\widetilde{S}$, let $W$ be a neighbourhood of $p$ in $\widetilde{S}$ such that the restriction of $\cD$ on $W$
%is injective. $\cD(W)$ is a smooth spacelike surface in $\R^{1,2}$, and reducing $W$ if necessary, $\cD(W)$ is the graph
%of a map $f_W \colon \bar{W} \rightarrow \R$ where $\bar{W}$ is the projection $\pi_0(\cD(W))$.

%Since $\cD_{\mid \widetilde{S}}$ is $\Gamma$-equivariant, we have, for every $\gamma$ in $\Gamma$:

%$$\forall \bar{p} \in \bar{W} \;\; f_{\gamma W}(\bar{\rho}(\gamma) \bar{p}) = f_W(\bar{p}) - T(\gamma)$$

%where $T(\gamma)$ is the vertical part of the translation part of $\rho(\gamma)$. Hence:

%$$(\bar{\rho}(\gamma))^*df_{\gamma W} = df_W$$

%Hence, all the pull-backs $\overline{\cD}^*(df_W)$ define altogether a closed
%1-form on $\widetilde{S}$ which is $\Gamma$-invariant, hence is the lift of
%a closed 1-form $\omega$ on $S$.

\end{proof}

\begin{remark}
As an  immediate corollary of this is proposition, we obtain that the stationary flat singular spacetimes $(M, ds^2_\omega)$ are maximal; i.~e.~they
admit no proper isometric extensions.
\end{remark}

\section{Observers, particles and null geodesics}

In this section, we illustrate how our description of  stationary flat globally hyperbolic spacetimes with particles allows one to clarify their  causality properties and to derive general results
about the outcome of  measurements by observers.
These issues play an important role in the physical interpretation of the theory.
The question if spacetimes containing particles with spin  are admissible models in (2+1)-dimensional general relativity or if
the presence of closed causal  curves makes them unsuitable has been subject to much debate in the physics community.

Generally, the presence of closed timelike curves (CTCs) in a spacetime is problematic from the physics point of view because it leads to  ``grandfather paradoxes''.
Each  point in the spacetime corresponds to a physical event, and a timelike
curve  defines the set of  the events experienced by an observer.  Given two points  $p,q$ on a {\em closed} timelike curve, it becomes impossible to determine which of the associated events happened to the observer before the other one, since - by definition -
each of them lies in the future and in the past of the other  one.

In the models under consideration, the CTCs are confined to a small region around each particle. Although this region is not inaccessible or hidden behind a horizon, one is tempted to argue
that the presence of CTCs in a small region around the particles is unproblematic if one intends to model the large scale behaviour of the spacetime and is interested only in observers that are located at a sufficient distance from the particles.

However,  this reasoning is too naive because it clashes with one of the fundamental notions of general relativity, namely the principle that observers can communicate with each other by  exchanging light signals. Such light signals are modelled by future directed null geodesics, which can enter the  regions in which closed timelike curves occur and remerge from them, even if the associated observers are located at a large distance from such regions. This can lead to light signals which are received by observers before they are emitted and again give rise to causality paradoxes.

In order to establish if flat globally hyperbolic spacetimes with particles are admissible physics models, it  is therefore necessary to analyse carefully and in detail how the presence of particles with spin affects light signals passed between different observers.
However, to our knowledge there is no systematic investigation of this issue in the literature.
In the following, we will show how this question can be resolved  for the stationary flat globally hyperbolic spacetimes with particles.

\subsection{Null geodesics in  the one-particle model}

To illustrate how the presence of particles with spin manifests itself in the measurements of observers,  we start with an informal discussion based on the
one-particle model introduced in Section \ref{singlepart}.
%with an angle  $\theta_0$ and spin $\sigma$.
As in Section \ref{singlepart}, we denote by $\theta_0=2\pi-\mu$ the associated apex angle, by $\sigma$ the spin of the particle, and assume that the  singular line associated with the particle is given in radial coordinates by the equation $r=0$.

We consider observers who probe the geometry of the spacetime by emitting and receiving lightrays.  For simplicity, we will restrict attention to observers which do not undergo acceleration, and we will require that they do not collide with the particle.   Such observers are characterised uniquely by a worldline, a future directed timelike geodesic in the complement of the singular line.

A lightray is modelled  by  a future directed null geodesic in $M$. Note that we will not require that this null geodesic avoids the CTC region or the singular line, which in physical terms means that light can pass through the particle. Moreover, the lightrays under consideration are ``test lightrays''  in analogy to  the test masses  used in thought experiments on general relativity. This means we consider them as hypothetical lightrays in a given spacetime and neglect their contribution to the stress-energy tensor.

In the presence of particles, it  is possible that a lightray that is emitted by an observer at a given time returns to the observer.  Such a returning lightray corresponds to a future directed null geodesic $\gamma:[0,1]\rightarrow M$ which intersects the observer's worldline twice. As shown in Section \ref{sub:geodesic}, such  geodesics exist only for large masses $\mu=2\pi-\theta_0>\pi$ which correspond to cones with apex angles $\theta_0<\pi$. We therefore restrict attention to this case.

To obtain an explicit description of  returning lightrays, we need a concrete parametrisation for the future directed timelike geodesic which characterises the observer.
Making use of the symmetry of the one-particle system under rotations around the $t$-axis, we assume that in rectangular coordinates $(x,y,t)$ on $\R^{2,1}$ this  worldline takes the form
$$
g(t)=(0,d,0)+t(\sinh\vartheta, 0 ,\cosh\vartheta)\qquad \vartheta,d\in \R.
$$
Note that the  parameter $d$ gives the minimal distance between the particle and observer as it appears in the reference frame of the particle.   The parameter $\vartheta$ defines the relative velocity $v$ of observer relative to the particle: $v=\tanh\vartheta$. The parameter $t$ coincides with the eigentime of the observer, i.~e.~ the time shown on a clock carried by the observer. Note that the origin of this time variable is chosen in such a way  that the distance between the observer and the particle is minimal at $t=0$.

In the universal cover, returning lightrays are characterised by future directed null geodesic segments whose endpoints lie on two different lifts of the observer's worldline as indicated
The lifts of the observer's worldline to the universal cover can be parametrised as
$$
g_m(t)=R^m g(t)+m\sigma(0,0,1),
$$
where $m\in\Z$ and $R$ is the rotation around the $t$-axis by an angle $\theta_0$. It follows from the discussion in Section \ref{sub:geodesic} that
a null geodesic connecting two different lifts  $g_m$ and $g_n$ exists if and only if $0<|m-n|<\left|\frac \pi {\theta_0}\right|$.
A light signal emitted by the observer at eigentime $t-\Delta t$ and received at eigentime $t$
corresponds thus to a future directed null geodesic  from $g_m(t-\Delta t)$ to  $g_0(t)$ with $0<|m|<\left|\frac \pi {\theta_0}\right|$. Such a returning light ray  is depicted in Figure \ref{wedgecover}.

 \begin{figure}
  %\centering
  \includegraphics[scale=0.4]{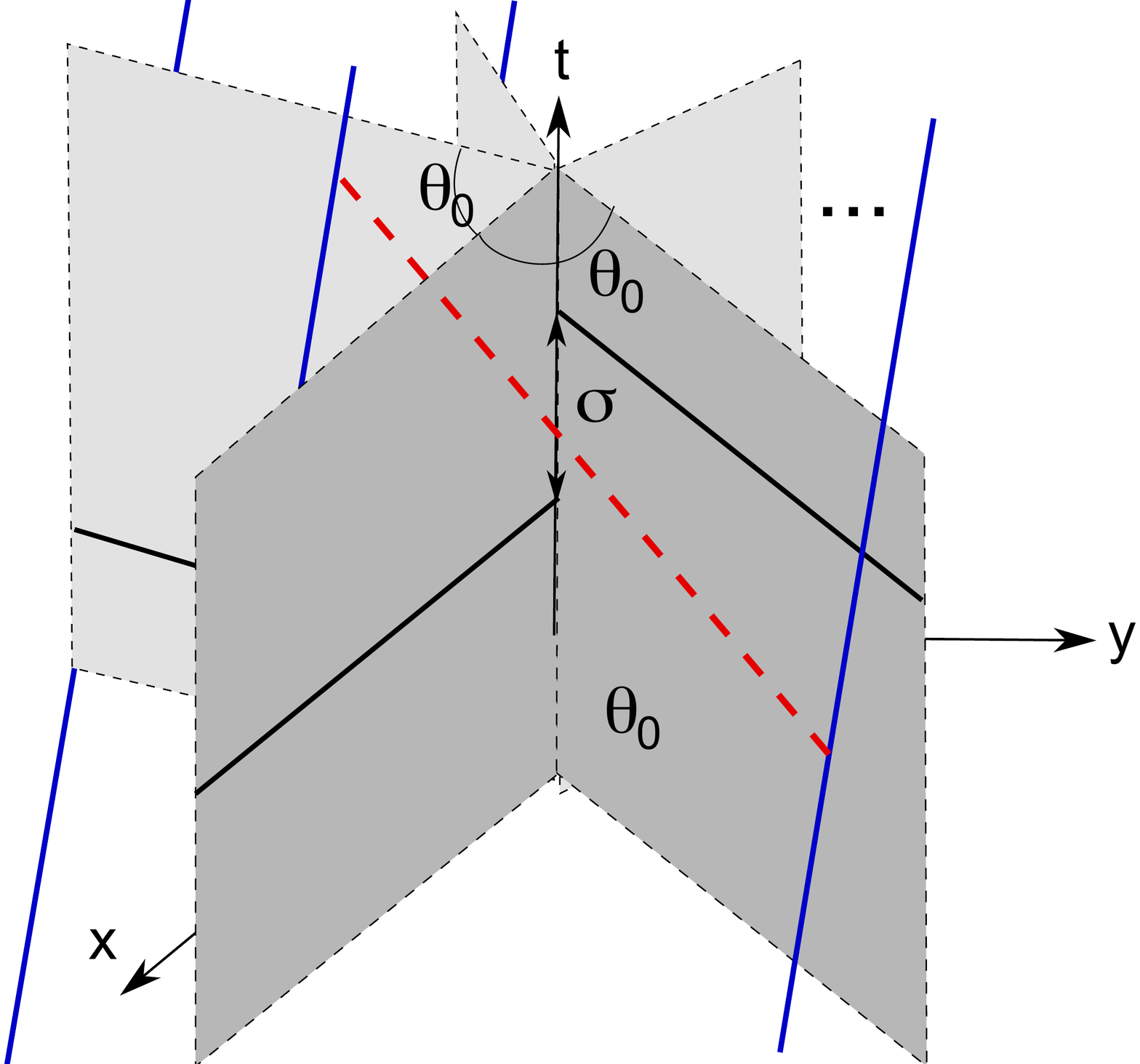}
\centering  \caption{Returning light ray in the universal cover. The blue solid  lines correspond to different lifts of the observer's worldline, the red dashed line depicts a returning light signal.}
  \label{wedgecover}
\end{figure}

 The  requirement that  $g(t)-g_m(t-\Delta t)$ is lightlike defines a quadratic equation in $\Delta t$.  A short calculation shows that the solution of this equation that  corresponds to a future-directed lightray is given by
\begin{align}\label{rett}
&\Delta t_m=m\sigma\cosh\vartheta+\sinh\vartheta T_m
+\sqrt{(m\sigma\sinh\vartheta+\cosh\vartheta T_m)^2+S_m^2}
\end{align}
where
$$
T_m=2\sin\tfrac{m\theta_0} 2\left(t\sinh\vartheta\sin\tfrac{m\theta_0} 2+d\cos\tfrac{m\theta_0} 2\right)\qquad S_m=2\sin\tfrac{m\theta_0} 2\left(t\sinh\vartheta\cos\tfrac{m\theta_0} 2-d\sin\tfrac{m\theta_0} 2\right).
$$
This defines the return time of the signal, the interval $\Delta t_m$ of eigentime elapsed between the emission and the reception if the lightray. Note, however, that depending on the sign of the spin and on the direction into which the lightray travels around the particle, the first and second term in this formula can become negative.  A short calculation shows that the return time $\Delta t_m>0$ if and only if
\begin{align}\label{cond}
\sqrt{t^2\sinh^2\vartheta+d^2}> \left|\frac{m\sigma}{2\sin\tfrac{m\theta_0}{2}}\right|.
\end{align}
From the concavity of the sine function it follows that this condition is satisfied for all admissible  values of $m$ if
\begin{align}\label{globcond}
\sqrt{t^2\sinh^2\vartheta+d^2}> \frac{\pi |r_0|}{2}
\end{align}
Conditions \eqref{cond}, \eqref{globcond} have a direct physical interpretation. The term $\sqrt{t^2\sinh^2\vartheta+d^2}$ on the left-hand-side gives the distance of the observer from the particle  at the  moment at which the observer receives the returning lightray and with respect to the momentum rest frame of the particle. If this distance is smaller than the term on the right-hand side, the effect of the spin can become dominant and results in a negative return time: $\Delta t_m<0$. This corresponds to a
light signal that is received before it is emitted thus violating causality. Note that causality violating signals  can arise even if the observer does not enter
 the CTC zone associated with the particle.  Similarly,  it is irrelevant if  the observer had entered or will enter the CTC zone at an earlier or later time. What determines if causality violating light signals can be received at a given time is the distance of the observer from the particle at {\em that time}.

In order to exclude causality violating light signals for all times $t$ and values of $m$, one therefore needs to impose that the minimum distance $|d|$ of the observer from the particle viewed from the reference frame of the particle satisfies the condition
\begin{align}\label{minc}
|d|> \frac {\pi |r_0|} 2=:r_c.
\end{align}
Note that this condition represents a considerable weakening with respect to the condition used in the definition of global hyperbolicity.
In Definition~\ref{def:spinGH} only causal curves which {\em do not} enter the CTC zones around the particles are admissible. In this example, we consider a closed, piecewise geodesic causal curve (obtained by composing the light ray with the segment of the timelike geodesic which characterises the observer). This curve is such that the the closest point of the timelike segment lies
 outside of a circle with radius $r_c$ around the particle but whose lightlike segment may enter the CTC zone.

It is instructive to consider how measurements with returning light signals allow the observer to obtain information about the spacetime, i.~e.~ to determine the position, mass and spin of the particle.
For simplicity we restrict attention to observers which are at rest with respect to the particle. Such observers are characterised by the condition $\vartheta=0$. For such an observer, determining the return time reduces to a two dimensional problem that can be solved by elementary geometry.
The return time \eqref{rett} takes the form
\begin{align}\label{ret2}
\Delta t_m=m\sigma+2d\left |\sin\tfrac{m\theta_0} 2 \right| =2m+l,
\end{align}
where  $l$ is   the length of the straight line in  the $t=0$-plane that connects its intersection point with the observer's worldline and and its image under the rotation $R$ as shown in Figure \ref{timepic}.
 \begin{figure}
  %\centering
  \includegraphics[scale=0.4]{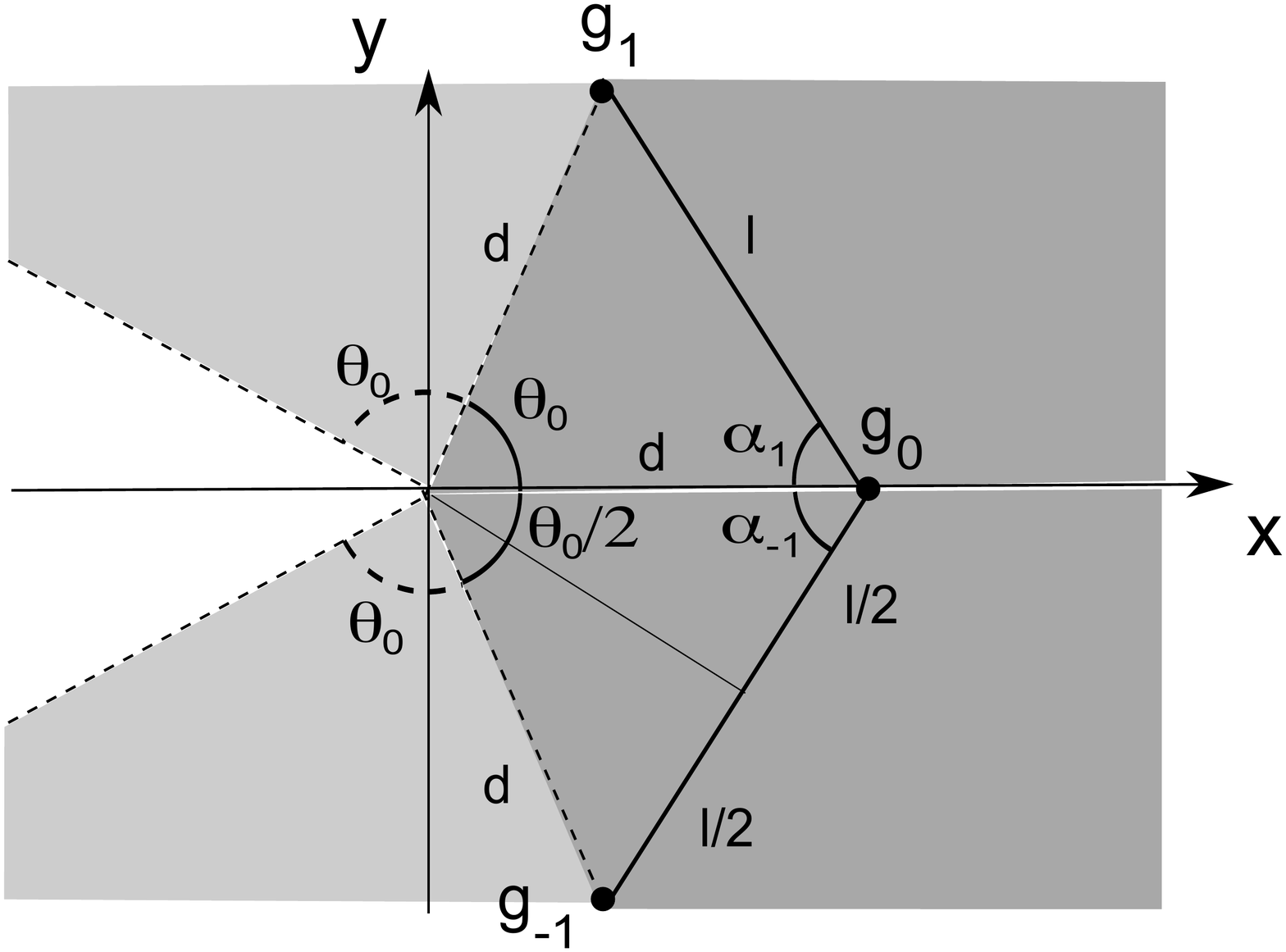}
\centering  \caption{Different lifts of the  observer and the projections of returning null geodesics. The points $g_0,g_1,g_{-1}$ correspond to the intersection points of the lifted worldlines with the plane characterised by $t=0$. The solid lines correspond to the projections of returning lightrays.}
  \label{timepic}
\end{figure}
The directions from which the observer receives the returning lightrays are given by the vertical projection of the lightlike vector $g(t)-g_m(t-\Delta t_m)$ on the plane $t=0$-plane.
The angle $\alpha_m$ between the direction of the particle as viewed by the observer and the direction of the returning lightray is therefore given  by
\begin{align}\label{helpid}
\alpha_m= \frac {\text{sgn}(m) (\pi - |m|{\theta_0})} 2
\end{align}
%and the eigentime elapsed between the emission and the return of the lightray  takes the form
%$$
%\Delta t_m=m\sigma+2d|\cos\alpha_m|.
%$$
This allows the observer to draw conclusions about the location of the particle and to determine its mass $\mu=2\pi-\theta_0$ and spin $\sigma$. For this it is sufficient that the observer determines the direction and the return time for the first two lightrays which are given by $m=\pm 1$.
The observer then obtains the direction of the particle  by constructing the bisector of the angle between the directions of these two  returning lightrays. A measurement of  this angle  allows him to determine the particle's mass via \eqref{helpid}. The
 time elapsed between the return of the two lightrays yields the particle's spin $\sigma$ via \eqref{ret2}, and by taking the sum $\Delta t_1+\Delta t_{-1}=4d\cos\alpha_1$, the observer obtains his distance $d$ from the particle.
By emitting and receiving returning  lightrays, the observer can thus determine all parameters associated with the model: the position of the particle, its mass $\mu=2\pi-\theta_0$ and its spin $\sigma$.

\subsection{Null geodesics in stationary flat spacetimes with particles}

We will now generalise our discussion from the previous subsection to general stationary flat globally hyperbolic spacetimes with particles
 and to more general light signals emitted and received by several observers.
Let $(M, ds^2_\omega)$ be a stationary flat spacetime with particles as in Definition \ref{omspt}, which is globally hyperbolic, i.~e.~satisfies the conditions of Proposition~\ref{pro:ghspin}.

The discussion from Section \ref{statsec} shows that there are no closed causal curves contained in the interior region. However,  by definition, there is a region around each particle which the metric takes the same form as the one-particle model. If the apex angles associated with the particles are sufficiently small, this allows for
 the existence of light signals that enter the CTC region and that return to the observer before they are emitted. In analogy to the one-particle model, one can show that such signals are not present if the observer is sufficiently far from each particle.

However, this does not address the issue of causality violating signals in full generality.
In general relativity, it is  traditional to consider light signals that are passed back and forth between several different observers.
 Such light signals play a  fundamental  role in the physical interpretation of the theory since they allow observers to synchronise clocks and to measure distances and relative velocities.

To demonstrate that  these spacetimes have acceptable causality properties, we therefore need to establish physically reasonable conditions that ensure that measurements conducted by a ``team of observers''  located in the interior region of the spacetime do not result in light signals that return to an observer before they are emitted.
As we will show in the following, this is guaranteed if one imposes that all of the observers are located sufficiently far away from each particle.

To derive this result, we  give a precise definition  of the relevant  physical  concepts of observers, lightrays and light signals. For simplicity, we will restrict attention to observers whose worldlines are parallel to the singular lines that define the particles.

\begin{defi} [Observers and light signals] \label{obs}$\quad$\\
Let $(M, ds^2_\omega)$ be a stationary globally hyperbolic  flat spacetime  with particles, and denote by $M^*$ its regular part, i.~e.~ the complement of its singular lines.

\begin{enumerate}

\item A {\em stationary observer}  in $M$ is a future directed  vertical geodesic in $M^*$.
\item  A light ray sent from a stationary
 observer $g_1$ to a stationary  observer $g_2$ in $M$ is a future directed null geodesic  $\gamma:[0,1]\rightarrow M$ with $\gamma(0)\in g_1$ and $\gamma(1)\in g_2$.
 \item A {\em light signal} sent from a stationary  observer $g_1$ to a stationary observer $g_2$ in $M$ is a piecewise geodesic curve $\gamma:[0,1]\rightarrow M$ with $\gamma(0)\in g_1$ and $\gamma(1)\in g_2$ for which there exists a subdivision
$[0,1]=\bigcup_{j=0}^N [s_j,s_{j+1}]$, $s_0=0$, $s_{N+1}=1$ such that $\gamma|_{[s_j,s_{j+1}]}$ is a future directed null geodesic.
\end{enumerate}
\end{defi}
By identifying $M$ with $\Sigma \times\R$,  one can express every light signal  $\gamma:[0,1]\rightarrow M$ from $g_1$ to $g_2$ as a function $\gamma(s) = (c(s), t(s))$ where $c: [0,1] \to \Sigma$ is a  piecewise geodesic curve on $\Sigma$ with $c(0)=g_1\cap \Sigma$ and $c(1)=g_2\cap\Sigma$. Note that  the function $t:[0,1]\rightarrow \R$ is characterised uniquely up to a global constant by the requirement that $\gamma$  can be subdivided into future directed null geodesics.

From the definition, it is clear that the concept of a light signal encompasses precisely the situation discussed above. Each  point $c(s_j)$ on $\Sigma$ for $j\in\{0,...,N+1\}$ corresponds to an observer,  which is given by the unique vertical line through $c(s_j)$. The first observer at $c(0)$ emits a light ray at $t(0)=0$, that  is described by the null geodesic  $\gamma|_{[s_{0},s_{1}]}$. The second observer at $c(s_1)\in\Sigma$ receives this lightray at $t(s_1)$
and immediately emits another lightray, $\gamma|_{[s_1,s_2]}$, which is received by the third observer at $t(s_2)$ and so on, until the last lightray is received by the observer  at $t(1)$.  This situation is depicted in Figure \ref{observers}.
Piecewise geodesic curves on the surface $\Sigma$ thus have a natural general relativistic interpretation: they define groups of observers that transmit a signal by sending and receiving lightrays.

The question is now if by passing such light signals between different observers, it is possible to construct a light signal that returns to the first observer  and which is such that $t(1)<t(0)$. As we can identify $t(1)-t(0)$ with the time elapsed between the emission of the light signal and its reception as shown on a clock carried by this observer, the condition $t(1)<t(0)$ states that  the light signal is received before it is emitted, which is an obvious violation of causality.

\begin{defi} Let $(M, ds^2_\omega)$ be a stationary globally hyperbolic  flat spacetime  with particles. A {\em returning light signal} for a stationary observer $g$ is a light signal from $g$ to $g$. It can be expressed as a curve $\gamma:[0,1]\rightarrow M$, $\gamma(s)=(c(s),t(s))$, where $c: [0,1] \to \Sigma$ is a closed piecewise geodesic curve with $c(0)=c(1)$.
The light signal  is called  a {\em paradoxical} if $t(1)<t(0)$.
\end{defi}

 \begin{figure}
  %\centering
  \includegraphics[scale=0.4]{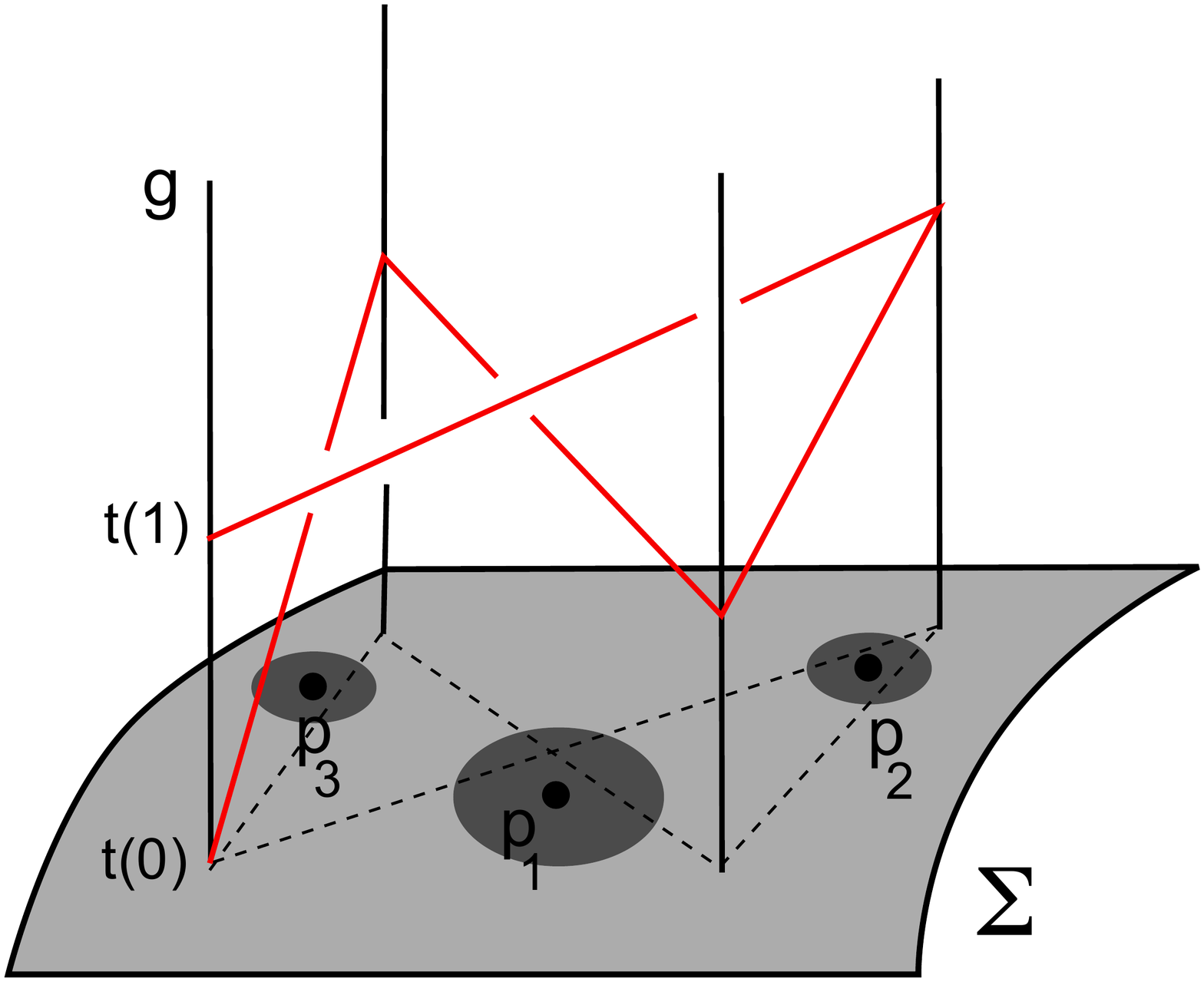}
\centering  \caption{Returning light signal in a stationary flat spacetime $(M,ds^2_\omega)$. The black solid lines depict the worldlines of stationary observers in $M$, the red solid line corresponds to a returning light signal. The shaded ares around the points on $\Sigma$ correspond to the CTC regions around each particle, and the grey dashed line to the associated piecewise geodesic curve on $\Sigma$.}
  \label{observers}
\end{figure}

We will now show that there is  direct and intuitive  condition that rules out paradoxical light signals  and which depends only on the location of the observers.  In other words, as long as the observers do not come too close to the particles, there is no possibility of paradoxical light signals.
This implies that although causality is violated near the particles, observers which are located at a certain distance from them do not encounter signals that are received before they are emitted.

\begin{prop}\label{sevobs} Let $(M, ds^2_\omega)$ be a stationary globally hyperbolic spatially compact flat spacetime with particles defined in terms of a Euclidean metric $ds^2_0$
with cone singularities  and a closed $1$-form $\omega$ on a closed surface $\Sigma$. For every singular line  $d_i$, denote by $r^i_c:=\frac\pi 2 |r_i|=\pi |\sigma_i|/2\theta_i$ %where $|r_i|=|\sigma_i|/\theta_i$ is
 the CTC radius rescaled by $\pi/2$. Let $\gamma: [0,1]\rightarrow M$, $\gamma(s)=(c(s), t(s))$ be %a piecewise geodesic curve such that $c(a)$ and $c(b)$ lie on a vertical line and for which there exists a subdivision
%$[0,1]=\bigcup_{j=0}^N [s_j,s_{j+1}]$, $s_0=0$, $s_{N+1}=1$ such that $\gamma|_{[s_j,s_{j+1}]}$ is a future directed null geodesic.
a returning light signal given in terms of a closed piecewise geodesic curve $c: [0,1] \to \Sigma$ with a subdivision as in Definition \ref{obs}.
Assume that $c$ is not constant and \begin{equation}\label{distance}
d_{ds_0^2}(c(s_j),p_i)>r^i_c \qquad \forall j\in\{0,...,N+1\}, i\in\{0,...,n\}.
\end{equation}
Then  $\gamma$ does not give rise to a  paradox, i.~e.~ it satisfies
 $t(1)-t(0)>0$.
\end{prop}

To prove the proposition, we note that its content can easily be reformulated as a statement on
closed, piecewise geodesic curves on the Euclidean surface with cone singularities.
We have:

\begin{remark} \label{rem}
Proposition \eqref{sevobs} is equivalent to the following statement.  Let $c:[0,1]\rightarrow \Sigma$ be a closed curve for which there exists a subdivision
$[0,1]=\bigcup_{j=0}^N [s_j,s_{j+1}]$, $s_0=0$, $s_{N+1}=1$ such that $c|_{[s_j,s_{j+1}]}$ is a geodesic, $c(s_0)=c(s_n)$  and
which satisfies condition \eqref{distance}.
Then:
\begin{equation}\label{omegal}
l(c)>\int_{c} \omega
\end{equation}
Indeed, combining this inequality with the corresponding inequality involving the $1$-form $-\omega$ we obtain:
$$
l(c)>\left|\int_{c} \omega\right|
$$
\end{remark}
\noindent By means of this remark, we can now give a direct proof of Proposition \ref{sevobs}:
\begin{proof}[Proof of proposition \ref{sevobs}]
For every cone singularity $p_i$, let $D_i$ be the disk of radius $r^i_c$ centred at $p_i$. We can assume that  up to a coboundary,
$\omega$ coincides in each disc $D_i$ with $\frac{\sigma_i}{\theta_i}d\theta=r_id\theta$, where $\theta$ is the angular coordinate
in the wedge of angle $\theta_i$.

According to Corollary \ref{cor:omega1}, one can  adjust $\omega$, without perturbing
the previous property, so that the integration of $\omega$ along a curve contained in the interior region - \textit{a fortiori,}
outside the disks $D_i$ - cannot exceed the length of the curve.%\ques{this needs to be explained in more detail and precision. I do not understand what you mean here by ``a fortiori'': just because the complement of the disks is contained in the interior region!}

Let $c:[0,1]\rightarrow \Sigma$ be a closed piecewise geodesic curve satisfying the assumptions of Remark \ref{rem}. Then
there is a subdivision of $[0,1]$ such that every interval $[a,b]$ of this subdivision is either contained in the complement of
the disks $D_i$, or in such a disk.

We have just seen that in the first case, we have the inequality:
\begin{equation}\label{omegall}
\int_{c|_{[a,b]}} \omega < l(c|_{[a,b]})
\end{equation}
Consider now the second case. In this case,
the endpoints $c(a)$ and $c(b)$ lie on the boundary $\partial D_i$.  As the corner points $c(s_i)$ lie outside $D_i$, the restriction
of $c$ to $[a,b]$ is a geodesic arc. Consequently, there is an integer $m$ such that the lift $\tilde{c}$ of $c|_{[a,b]}$ to the $m$-branched cover $D^m_i$ of $D_i$
 is a minimising geodesic arc. Choose a polar coordinate system $(r,\theta)$ in $D_i$ such that  $\theta(\tilde c(a))=0$
and let $\alpha$ be the angular coordinate of $\tilde{c}(b)$. According to Section~\ref{sub:geodesic}, we have $0 < \alpha < \pi$.
According to our choice of $\omega$, the lift $\tilde{\omega}$ of $\omega$ in $D^m_i$ is the variation of the angular coordinate multiplied by
the factor ${r_i}/{2\pi}$. This implies
$$
\int_{c|_{[a,b]}} \omega = \int_{\tilde{c}|_{[a,b]}} \tilde{\omega} = r_i\alpha
$$
As $\tilde{c}|_{[a,b]}$ is a chord of angle $\alpha$ of a circle of radius $r^i_c$, its length is given by
$$
l(c|_{[a,b]}) = l(\tilde{c}|_{[a,b]}) = 2r^i_c\sin\frac\alpha2={|r_i|}\pi\sin\frac\alpha2
$$
As $\frac\alpha2<\frac\pi2$, we have $\alpha < \pi\sin\frac\alpha2$. It follows that the inequality \eqref{omegall} also holds
in this case.
This implies that on each subinterval $[a,b]$ of the subdivision of $[0,1]$, the inequality \eqref{omegall} is satisfied. By summing over all subintervals, we obtain the
inequality \eqref{omegal}.
\end{proof}

\section{Outlook and conclusions}

In this article, we give a systematic investigation of the causality structure of  flat, stationary (2+1)-dimensional Lorentzian manifolds with particle singularities and clarify its implications in physics.  As these manifolds contain closed timelike curves, the usual methods established  for globally hyperbolic manifolds cannot be applied directly
but have to be replaced by suitable generalisations.

By introducing a generalised notion of global hyperbolicity adapted to manifolds with particle singularities,  we are able to classify all stationary flat Lorentzian (2+1)-dimensional manifolds with particles which are globally hyperbolic in that sense.
This classification result  characterises flat, stationary globally hyperbolic (2+1)-dimensional Lorentzian manifolds in terms of two-dimensional Euclidean surfaces with cone singularities and closed one-forms on these surfaces and thus provides an explicit and simple description.

It turns out that this description is particularly well-suited for the investigation of the causality structure of these manifolds from a physics  point of view.  It allows one to systematically
address the the question how the presence of massive point particles with spin manifests itself in measurements performed by observers in the spacetime.

We show how an observer in the spacetime can use the results of measurements with returning lightrays to determine the mass, spin, position and relative velocity of the particles and investigate more general light signals exchanged between several observers. It turns out that the latter have a natural interpretation in terms of piecewise geodesic loops on the underlying surface with conical singularities.

This allows us to derive a general condition on the observer that excludes paradoxical light signals which return to an observer before they are omitted.  In physics terms, our result implies that if all observers stay sufficiently far away from the particles, no causality violating light signals will occur, no matter how often the light signals exchanged between them enter spacetime regions which contain closed timelike curves.

It would be interesting to extend these results to more complete  classification of flat, globally hyperbolic  3d Lorentzian manifolds with particle singularities, which also take into account the non-stationary case.  As currently very little is known about these manifolds, a first step would be the construction and study  of relevant examples which generalise the examples currently known in the physics literature.  One could then attempt to classify these manifolds under suitable additional assumptions using the generalised notion of global hyperbolicity introduced in this paper.

\bibliography{bibliospin}
\bibliographystyle{alpha}

\end{document}

%% file: macros.tex
\newtheorem{cor}{Corollary}[section]
\newtheorem{theorem}[cor]{Theorem}
\newtheorem{prop}[cor]{Proposition}
\newtheorem{lemma}[cor]{Lemma}
\theoremstyle{definition}
\newtheorem{defi}[cor]{Definition}
\theoremstyle{remark}
\newtheorem{remark}[cor]{Remark}
\newtheorem{example}[cor]{Example}

\newcommand{\wT}{{\widetilde{\mathcal T}}}
\newcommand{\cB}{{\mathcal B}}
\newcommand{\cC}{{\mathcal C}}
\newcommand{\cD}{{\mathcal D}}
\newcommand{\FF}{{\mathcal F}}
\newcommand{\cH}{{\mathcal H}}
\newcommand{\cL}{{\mathcal L}}
\newcommand{\cG}{{\mathcal G}}
\newcommand{\cM}{{\mathcal M}}
\newcommand{\cT}{{\mathcal T}}
\newcommand{\cML}{{\mathcal M\mathcal L}}
\newcommand{\cGH}{{\mathcal G\mathcal H}}
\newcommand{\C}{{\mathbb C}}
\newcommand{\N}{{\mathbb N}}
\newcommand{\R}{{\mathbb R}}
\newcommand{\X}{{\mathbb X}}
\newcommand{\Z}{{\mathbb Z}}
\newcommand{\Kt}{\tilde{K}}
\newcommand{\Mt}{\tilde{M}}
\newcommand{\Mw}{\widetilde{M}}
\newcommand{\Sw}{\widetilde{S}}
\newcommand{\dr}{{\partial}}
\newcommand{\tr}{\mbox{tr}}
\newcommand{\isom}{\mbox{Isom}}
\newcommand{\isomz}{\mbox{Isom}_{0,+}}
\newcommand{\vect}{\mbox{Vect}}
\newcommand{\kappab}{\overline{\kappa}}
\newcommand{\pib}{\overline{\pi}}
\newcommand{\Sigmab}{\overline{\Sigma}}
\newcommand{\gd}{\dot{g}}
\newcommand{\diff}{\mbox{Diff}}
\newcommand{\dev}{\mbox{dev}}
\newcommand{\devb}{\overline{\mbox{dev}}}
\newcommand{\devt}{\tilde{\mbox{dev}}}
\newcommand{\vol}{\mbox{Vol}}
\newcommand{\hess}{\mbox{Hess}}
\newcommand{\db}{\overline{\partial}}
\newcommand{\gammab}{\overline{\gamma}}
\newcommand{\Sigmat}{\tilde{\Sigma}}
\newcommand{\mut}{\tilde{\mu}}
\newcommand{\phit}{\tilde{\phi}}

\newcommand{\D}{\mathbb D}

\newcommand{\cunc}{{\mathcal C}^\infty_c}
\newcommand{\cun}{{\mathcal C}^\infty}
\newcommand{\dd}{d_D}
\newcommand{\dmin}{d_{\mathrm{min}}}
\newcommand{\dmax}{d_{\mathrm{max}}}
\newcommand{\Dom}{\mathrm{Dom}}
\newcommand{\dn}{d_\nabla}
\newcommand{\ded}{\delta_D}
\newcommand{\delmin}{\delta_{\mathrm{min}}}
\newcommand{\delmax}{\delta_{\mathrm{max}}}
\newcommand{\hmin}{H_{\mathrm{min}}}
\newcommand{\maxi}{\mathrm{max}}
\newcommand{\oL}{\overline{L}}
\newcommand{\oP}{{\overline{P}}}
\newcommand{\Ran}{\mathrm{Ran}}
\newcommand{\tgamma}{\tilde{\gamma}}
\newcommand{\cotan}{\mbox{cotan}}
\newcommand{\lambdat}{\tilde\lambda}
\newcommand{\St}{\tilde S}

\newcommand{\II}{I\hspace{-0.1cm}I}
\newcommand{\III}{I\hspace{-0.1cm}I\hspace{-0.1cm}I}
\newcommand{\HSt}{\tilde{\operatorname{HS}}}
\newcommand{\note}[1]{\marginpar{\tiny #1}}

\newcommand{\op}{\operatorname}

\newcommand{\AdS}{\operatorname{AdS}}
\newcommand{\uAdS}{\widetilde{\operatorname{AdS}}}
\newcommand{\dS}{\operatorname{dS}}
\newcommand{\HH}{\mathbb H}
\newcommand{\PP}{\mathbb P}
\newcommand{\RR}{\mathbb R}
\newcommand{\uRP}{\widetilde{\R\PP}^1}
\newcommand{\HS}{\operatorname{HS}}
\newcommand{\SO}{\operatorname{SO}}
\newcommand{\cF}{\operatorname{\mathcal{F}}}
\newcommand{\kD}{\mathfrak{D}}